\documentclass[]{amsart}
\usepackage{amsmath,amsthm,amssymb,graphicx,enumerate,natbib,color}

\newtheorem{theorem}{Theorem}[section]
\newtheorem{lemma}[theorem]{Lemma}

\newtheorem{condition}[theorem]{Condition}
\newtheorem{assumption}[theorem]{Assumption}
\newtheorem{remark}[theorem]{Remark}
\newcommand{\defn}{:=}
\newcommand{\one}{{\bf 1}}
\def\bbr{\mathbb{R}}

\def\vep{\varepsilon}
\def\reals{\mathbb{R}}
\def\bbz{\mathbb{Z}}

\def\argmin{\text{argmin}}

\numberwithin{equation}{section}

\begin{document}

\author[S. Ghosh]{Souvik Ghosh}
\address{Department of Statistics \\
Columbia University \\
1255 Amsterdam Avenue, 1011 SSW \\
New York, NY 10027}
\email{ghosh@stat.columbia.edu}
\author[G. Samorodnitsky]{Gennady Samorodnitsky}
\address{School of Operations Research and Information Engineering\\
and Department of Statistical Science \\
Cornell University \\
Ithaca, NY 14853}
\email{gennady@orie.cornell.edu}
\title[Long Strange Segments and Ruin Probabilities]{Long Strange
  Segments, Ruin Probabilities and the  Effect of Memory on   Moving
  Average Processes}  

\thanks{Research partially supported by  NSA grant MSPF-05G-049, ARO
grant W911NF-07-1-0078 and NSF training grant ``Graduate and Postdoctoral
Training in Probability and Its Applications''  at Cornell University. Souvik Ghosh's research was also partially supported by the Faculty Research Allowance Program at Columbia University.}

\begin{abstract}
We obtain the rate of growth of long strange segments and the rate of
decay of 
infinite horizon ruin probabilities for a class of infinite moving
average processes with exponentially light tails. The rates are
computed explicitly. We show that the rates are very similar to those
of an i.i.d. process as long as the moving average coefficients decay fast
enough. If they  
do not, then the rates are significantly different. This demonstrates
the change in the length of memory in a moving average process
associated with certain changes in the rate of decay of the
coefficients. 
\end{abstract}

\maketitle

\begin{section}{Introduction} \label{sec:intro}

How does the length of memory in a stationary stochastic process
affect the 
behavior of important characteristics of the process such as the rate of
increase of the \emph{long  strange segments} and the rate of decay
of the \emph{ruin  probabilities}?  From a different point of view:
can one use such important characteristics of a stationary process to
tell whether or not the process has long memory. In this paper such
questions are discussed  for a class of $\bbr^d$-valued 
infinite moving average processes with exponentially light
tails. These are processes of the form 
\begin{equation} \label{maproc}
    	X_n=\sum_{i\in\mathbb{Z}}\phi_i Z_{n-i}, \  n\in \mathbb{Z}, 
\end{equation}
where $(Z_i,i\in \mathbb{Z})$ are i.i.d., centered, random vectors
taking values in $\mathbb{R}^d$. We assume existence of some
exponential moments, i.e. 
\[
	\text{there exists}\  \epsilon>0 \mbox{ such that }
	\Lambda(t):=\log E \left[ e^{t Z_0}  \right]<\infty\ \text{for
	all } t\in \mathbb{R}^{ d} \mbox{ with }	|t|<\epsilon. 
\]

Such a process, also known as a linear process (see
\cite{brockwell:davis:1991}), is well defined  if the coefficients are
square summable: 
\begin{equation}\sum_{i=-\infty}^\infty
   \phi_i^2<\infty.
\end{equation} 
If the stronger condition of absolute summability of the coefficients
holds, namely 
\begin{equation}\label{absum}
	\sum_{i\in\mathbb{Z}}|\phi_i|<\infty,
\end{equation}
then it is often said that  the process has short memory. This is mainly
because 
the covariances of the process are summable in this case, and
a process with absolutely summable covariances is often considered to
have short memory, see e.g. \cite{samorodnitsky:2006LRD}.  What about
other characteristics of a process, that are often more informative
than covariances? 

In a recent
article \cite{ghosh:samorodnitsky:2009} gave a complete picture of functional 
large, moderate and huge deviations for the moving average process and
discussed the effect of memory on them. In this paper we follow up by
obtaining the rate of growth of long strange segments and the rate of  decay
of the ruin  probabilities for the  moving average processes. We
consider two cases: one where the coefficients of the process are
absolutely summable, i.e. \eqref{absum} holds, and the other when
\eqref{absum} fails and the coefficients are balanced regularly
varying. We show that the rates are significantly  
different in these two cases. We view these results as showing the
effect of memory as well 
as indicating that the processes with absolutely summable coefficients
can be legitimately called short memory processes, while the
alternative family of processes can be legitimately viewed as a family
of long memory processes. 

We now define precisely that characteristics of a process that we will
study in this paper. Suppose that  $(X_n,n\in \mathbb{Z})$ is  a zero
mean  $\mathbb{R}^d-$valued,
stationary and ergodic stochastic process. Given any measurable set
$A\subset \mathbb{R}^d$, the lengths of the long strange segments are
random variables, defined as  
$$
	R_n(A)\defn \sup \Big\{j-i:0\le i<j\le n, \frac{S_j-S_i}{j-i}
        \in A\Big\}\,,  
$$
where $S_k=X_1+\cdots+X_k$ are the partial sums. That is, $R_n(A)$ is
the maximum length of a segment from the first $n$ observations whose 
average is in $A$. 
To understand the justification for the name long strange segments,
consider  any set $A$ bounded away from the origin  (that is  $0\notin
\bar A,$ where $\bar A$ is the closure of $A$.) Since the process is
ergodic, we would not expect the average value of the process
over a long time segment to be 
in $A,$ and it is strange if that happens.   If we use the process to
model a system, then the long strange segments are the time
intervals where the system runs at a different ``rate'' than 
anticipated, and it is of obvious interest to know how long such
strange intervals could be. 

The easiest way to see the connection between the long strange
segments and large deviations is by defining 
$$
	T_n(A)\defn \inf\Big\{l:\ \text{there exists}\  k,0\le k\le
        l-n,\frac{S_l-S_k}{l-k}\in A \Big\}; 
$$
$T_n(A)$ is the minimum number of observations required to have a
segment of length at least $n$, whose average is in the set $A$. It is
elementary to check that there is a duality relation between
the rate of growth of $T_n$ and the rate of growth of $R_n.$
Furthermore,  for any sequence $(X_n)$ of random vectors, 
\begin{equation}\label{iid1}
	- \limsup_{ n \rightarrow \infty} \frac{1}{n}\log P \left[
          S_n/n\in A  \right]  \le \liminf_{ n \rightarrow \infty}
        \frac{1}{n}\log T_n(A), \ \text{P-a.s.}
\end{equation}
and, if  $(X_n)$ are i.i.d., then also 
\begin{equation}\label{iid2}
	- \liminf_{ n \rightarrow \infty} \frac{1}{n}\log P \left[
          S_n/n\in A  \right]   \ge \limsup_{ n \rightarrow \infty}
        \frac{1}{n}\log T_n(A), \ \text{P-a.s.;}
\end{equation}
see e.g. Theorem 3.2.1 in \cite{dembo:zeitouni:1998}. In Section
\ref{sec:lrs} we exploit the connection between a general version of
long strange segments and large deviations to establish the rate of
growth of the 
long strange segments for the two classes of moving average processes
we are considering. We will observe a marked change (or
a phase transition) in the rate of growth when switching from one
family of moving averages to the other. 

The relations of the form \eqref{iid1} and \eqref{iid2}   are referred
to as the Erd\"os-R\'enyi law; \cite{ErdRen:1970kx} proved asymptotics
for longest head runs in i.i.d. coin tosses. See
\cite{Gordon:1986p6959}, \cite{Arratia:1990p6736},
\cite{Novak:1992p6961}, \cite{gantert1998functional}  and
\cite{Vaggelatou:2003p6942} and the references therein for versions on
this result under various Markov chain settings.

We mention at this point that a different case of this problem 
was considered in
\cite{mansfield:rachev:samorodnitsky:2001} and
\cite{rachev:samorodnitsky:2001}, where the assumption of
certain finite exponential moments was replaced by the assumption of
balanced regular varying tails with exponent $-\beta<-1$. These papers
consider  
linear processes as in \eqref{maproc} in dimension $d=1$.  In particular,
\cite{mansfield:rachev:samorodnitsky:2001} showed that  if
\eqref{absum} holds, then for any $y>0$ and $x>0$
\begin{equation} 
P\big(a_n^{-1} R_n((y,\infty))\le x\big)\rightarrow
\exp(-C_sy^{-\beta} x^{-\beta}) 
\end{equation}
where   $(a_{ n})$ is a sequence  that does not depend on the moving
average coefficients, and it is 
regular varying at infinity with index $\beta^{-1}$
(see \cite{resnick:1987} or \cite{bingham:goldie:teugels:1987} for
details on regular variation). On the other hand, $C_s>0$ is a constant, which may depend
on the moving average coefficients. This rate of growth $a_n$ of the
long strange segments is the same as in the i.i.d. case, that results
when  choosing $\phi_0=1$ and $\phi_i=0$ for all  $i\neq 0$. In the
subsequent paper  
\cite{rachev:samorodnitsky:2001} considered the case when
(\ref{absum}) fails to hold, but the coefficients $(\phi_i)$ are 
balanced regular varying at infinity with exponent $-\alpha,$
satisfying $\max\Big\{ \frac{1}{\beta},\frac{1}{2}\Big\}<\alpha\le
1$. This means that there is a nonnegative function $\psi$ with 
\begin{equation}\label{longmemory}
 \psi\in RV_{-\alpha}, \mbox{ such that
  }\frac{\phi_n}{\psi(n)}\rightarrow p, \mbox{
  }\frac{\phi_{-n}}{\psi(n)} \rightarrow 1-p, \text{as $n\to\infty$}
\end{equation}
for some $0\leq p\leq 1$. 
Under this assumption,  for any $y>0$ and $x>0$,
\begin{equation}
P\big( b_n^{-1}R_n((y,\infty))\le x \big)\rightarrow
\exp(C_ly^{-\beta}x^{-\beta\alpha}), 
\end{equation}
for some sequence $(b_n)\in RV_{(\alpha\beta)^{-1}}$. Therefore, the
long strange segments now grow at the higher rate $(b_n)$. This phase
transition be taken as the evidence of long range dependence in the
moving average process under the regular variation \eqref{longmemory}
of the coefficients. A similar phenomenon can be observed in Section
\ref{sec:lrs} of the present paper.

The second topic that we consider in this article is that of the 
\emph{ruin probabilities}. If $(Y_n)$ is an $\bbr^d$-valued stochastic
process, and $A$ a measurable set in $\bbr^d$, an infinite horizon
ruin probability is a probability of the type 
\begin{equation} \label{e:ruin.1dim}
\rho(u; A)=\rho(u)=P \left[ Y_n\in uA, \mbox{ for some }n\ge 1  \right]. 
\end{equation}
The name ``ruin probability'' derives from the one-dimensional case
with $A=(1,\infty)$: if we interpret  $Y_n$ as 
 the total losses incurred by a company until time $n$, and $u$ is the
 initial capital of the firm, then  the event in \eqref{e:ruin.1dim}
 is the event that the
 company eventually goes bankrupt. Probabilities of the type are of
 interest in queuing theory as well; see e.g. \cite{asmussen:2003}.  

In the context of moving average processes, we will define 
\begin{equation}\label{eq:ruin}
	Y_n=\sum_{i=1}^nX_i-a_{ n}\mu,
\end{equation}
for some $\mu\in \bbr^d$, a sequence $ (a_{ n})$ increasing to $
\infty$, with 
$(X_{ n})$ the infinite moving average process \eqref{maproc}.  The
classical  Cram\'er-Lundberg Theory (see e.g. Section XIII.5 in 
\cite{asmussen:2003}) says that, in dimension $d=1$,  
if $ (X_{ n})$ are i.i.d., 
and $ (a_{ n})$ is a linear sequence then (under an additional
condition) there exist positive constants $ c$ and $ \theta$ such that  
\begin{equation} \label{e:cramer}
 	\rho(u)\sim   c e^{ -\theta u} \ \ \ \ \mbox{ as } u \to \infty.
\end{equation}
This  result was later extended by \cite{gerber:1982} to the situation
where $(X_{ n})$  an ARMA($ p,q $) process satisfying certain
assumptions, including that of bounded innovations, and
\cite{promislow:1991} has a further extension to certain infinite
moving average processes while  removing the assumption of the
boundedness of the innovations. In all these cases (\ref{absum}), which
we regard as a short memory case is assumed to hold (in fact, much 
stronger assumptions are needed). 

A weaker version of the estimate \eqref{e:cramer} is the logarithmic
scale estimate
\begin{equation} \label{e:log.cramer}
\lim_{u\to\infty}\frac1u \log\rho(u) = -\theta\,.
\end{equation}
Such results were derived in \cite{nyrhinen:1994,nyrhinen:1995} in a
fairly great generality in the one dimensional case. When specified to
the moving average case, in order to give a non-trivial limit, 
these results require, once again, absolute 
summability of the coefficients. 

There have been other recent studies of ruin probabilities for certain 
stationary increment processes with long memory. The papers 
\cite{husler:piterbarg;2004} and \cite{husler:piterbarg;2008} analyzed
the (continuous time) ruin  probability where the increment process was a version
of the fractional Gaussian noise. Further, \cite{barbe2008elf}
also obtained a 
logarithmic form of ruin probability asymptotics, as in \eqref{e:log.cramer}, under
the assumption that the increment process is the classical Fractional 
ARIMA process or belongs to a class of related processes.


In this paper we solve the logarithmic scale ruin problem
\eqref{e:log.cramer} when the increment  process $(X_{ n})$ in
\eqref{eq:ruin} is the infinite moving average process. We present a
fairly complete picture. Namely, we prove results both in the short
memory case (when (\ref{absum}) holds), and in the long memory case,
under the assumption of balanced regularly varying coefficients. We
allow a very broad class of  drift sequences $ (a_{ n})$. Ruin
probabilities are also related to large deviations, but not as
directly as the long strange segments. We use a
combination of multiple techniques, but  the large
deviation principle for the moving average process proved in
\cite{ghosh:samorodnitsky:2009} still plays an important role.  The
techniques we use here can 
modified for other, and more general, classes of stationary
processes but we do not make any such attempt in this paper. We
present the results and their 
proofs in Section \ref{sec:ruin} and in the process we clearly
demonstrate the effect of memory in the process $(X_{ n})$ on the rate
of the decay of the ruin probability $ \rho(u)$. The Appendix contains
a multivariate extension of the estimates in \cite{nyrhinen:1994} that
are not restricted to moving average processes. 
\end{section}

\begin{section}{Long Strange Segments}\label{sec:lrs}

Let $(X_n,n\in \mathbb{Z})$ be a $\mathbb{R}^d$-valued, centered
stationary infinite moving average  process \eqref{maproc} 
 defined on a probability space $(\Omega,\mathcal{F},P)$, and let
 $(S_n)$ be its partial sum process. 
  In this section we discuss the
 rate of growth of a general version of the length of the long strange
 segments, which we define as follows. For a  sequence
$\underline{a}=(a_n)$ increasing to infinity and a  measurable set
 $A\subset\mathbb{R}^d$, we define 
\begin{equation} \label{e:str.seg.gen}
R_m(A;\underline{a})\defn \sup\Big\{ n:\, \frac{S_l-S_{l-n}}{a_{n}}\in
A \ \text{for some $l=n,\ldots ,m$}\Big\}
\end{equation}
and the ``dual characteristic'' 
\begin{equation} \label{Trgen} 
	T_r(A;\underline{a})\defn \inf\Big\{l:\, \text{there exists}
	\ k,0\le k\le l-r,\frac{S_l-S_k}{a_{l-k}}\in A \Big\}. 
\end{equation}
Notice that $\{ R_m(A;\underline{a})\geq r\}$ if and only if $\{
T_r(A;\underline{a})\leq m\}$. 
We will often refer to $R_m(A;\underline{a})$ as $R_m$ and to
$T_r(A;\underline{a})$  as $T_r$, as long as the set $A$  and the
sequence $(a_n)$ under consideration are obvious. 

The assumptions and results below use the following notion of balanced
regular variation on $\bbr^d$.

A function
$f:\mathbb{R}^d\rightarrow \mathbb{R}$ is said to be  \emph{balanced
  regular varying} with exponent $\beta>0$, if there exists a
non-negative bounded function $\zeta_f$
defined on the unit sphere on $\mathbb{R}^d$ and a function
$\tau_f:(0,\infty)\rightarrow   (0,\infty)$ satisfying  
\begin{equation}\label{bregvar1}
\lim_{t\rightarrow \infty}\frac{\tau_f(tx)}{\tau_f(t)}=x^\beta
\end{equation} 
for all $x>0$ (i.e. $\tau_f$ is regularly varying with exponent
$\beta$)  and such that for any $(\lambda_t)\subset \mathbb{R}^d$ with
$|\lambda_t|=1$ for all $t$, converging to $\lambda$, 
\begin{equation}\label{bregvar2} 
\lim_{t\rightarrow \infty}\frac{
    f(t\lambda_t)}{\tau_f(t)}=\zeta_f(\lambda).
\end{equation} 
The subscript $f$ will typically be omitted if doing so is  unlikely to
cause confusion. 

Next, we state the specific assumptions on the moving average
process, the normalizing sequence $(a_n)$ in \eqref{e:str.seg.gen} and
\eqref{Trgen}, the resulting large deviations rate
sequence $(b_n)$, and the noise variables. We will consider two
different situations, corresponding to what we view as a short memory
moving average, when the coefficients in \eqref{maproc} decay fast,
and a long memory
moving average, when the coefficients in \eqref{maproc} decay
slowly. The Assumptions 2.1 and 2.2 below  correspond, roughly, to
Assumptions 2.1 and 2.3 in \cite{ghosh:samorodnitsky:2009}, respectively. 

We start with the assumptions describing the short memory
case. Throughout this paper we use $ \Lambda (\cdot)$ to denote the
log-moment generating function of the i.i.d. innovations $ (Z_{ i})$: 
\[ 
 	\Lambda (t):= \log E \big[  e^{ t Z_{ 0}} \big]\,,   
\] 
and by $ \mathcal{F}_{ \Lambda }\subset \mathbb{R}^{ d}$ we denote the
set where $ \Lambda(\cdot)$   is finite: 
\[ 
 	\mathcal{F}_{ \Lambda }=\{   t: \Lambda (t)< \infty  \}.   
\]
Furthermore, for any set $ A$, $ A^{ \circ}$ and $ \bar A$ denote the
interior and closure of $ A$, respectively.

\begin{assumption} \label{ass:short}
All the scenarios below assume that 
\begin{equation}\label{cabsum}
\sum_{i\in \mathbb{Z}}|\phi_i|<\infty
  \mbox{ and }\sum_{i\in \mathbb{Z}}\phi_i=1.
\end{equation} 
\begin{enumerate}[$S1.$]
\item $a_n=n,0\in \mathcal{F}_\Lambda^\circ$ and $b_n=n$.
\item $a_n=n, \mathcal{F}_\Lambda=\mathbb{R}^d$ and $b_n=n$.
\item $a_n/\sqrt{n\log n}\rightarrow \infty, \ a_n/n\rightarrow 0, \ 0\in
  \mathcal{F}_\Lambda^\circ$ and $(b_n)$ an increasing positive
  sequence such that $b_n\sim a_n^2/n$ as $n\to\infty$. 
\item $a_n/n\rightarrow \infty$, $\Lambda(\cdot)$ is balanced regular
  varying with exponent $\beta>1$ and $(b_n)$ an increasing positive
  sequence such that $b_n\sim n\tau(c_n)$, where  
\begin{equation}\label{cn} c_n=\sup \{x:\tau(x)/x\le
  a_n/n\}.
\end{equation} 
\end{enumerate}
\end{assumption}

The next assumption describes the long memory case. 
\begin{assumption}\label{regularvarying} \label{ass:long}
All the scenarios assume that the coefficients $(\phi_i)$ are
  balanced   regular varying with exponent $-\alpha,1/2<\alpha\le 1$
  and   $\sum\limits_{i=-\infty}^\infty|\phi_i|=\infty$. Specifically,
 we assume that \eqref{longmemory} holds for $\alpha$ in this range. 
Let $\Psi_n\defn \sum_{1\le i\le n}\psi(i)$, where once again, $ \psi(\cdot)$ is as  in \eqref{longmemory}. 
\begin{enumerate}[$R1.$]
\item $a_n=n\Psi_n,0\in   \mathcal{F}_\Lambda^\circ$ and $b_n=n$.
\item $a_n=n\Psi_n,   \mathcal{F}_\Lambda=\mathbb{R}^d$ and $b_n=n$.
\item $a_n/\bigl(\sqrt{n\log n}\Psi_n\bigr)\rightarrow \infty,
  a_n/(n\Psi_n)\rightarrow 
  0, 0\in \mathcal{F}_\Lambda^\circ$ and $(b_n)$ is an increasing
  positive sequence such that $b_n\sim a_n^2/(n\Psi_n^2)$ as
  $n\to\infty$. 
\item $a_n/(n\Psi_n)\rightarrow \infty,$ $\Lambda(\cdot)$ is balanced
  regular varying with exponent $\beta>1$ and $(b_n)$ is an increasing
  positive sequence such that   $b_n\sim n\tau(\Psi_nc_n)$, where  
\begin{equation}\label{cnp} c_n=\sup \{x:\tau(\Psi_nx)/x\le
  a_n/n\}.
\end{equation} 
\end{enumerate} 

\end{assumption}

Let $\mu_n(\cdot)\equiv\mu_n(\cdot;\underline{a})$ denote the law of
$a_n^{-1}S_n$. We quote the ``marginal version'' of the functional
results in \cite{ghosh:samorodnitsky:2009}; in certain cases these
have been known even earlier. The sequence $(\mu_n)$ satisfies the
large deviation principle on $\reals^d$:
\begin{equation}\label{e:LDP}
-\inf\limits_{x\in A^\circ}I_l(x) \leq
\liminf_{n\rightarrow\infty}\frac{1}{b_n}\log \mu_n(A;\underline{a})
\leq \limsup_{n\rightarrow\infty}\frac{1}{b_n}\log
\mu_n(A;\underline{a}) \leq -\inf\limits_{x\in \bar{A}}I_u(x)
\end{equation}
with a good lower function $I_l$ and a good upper function $I_u$ given
by 
\begin{eqnarray} \label{e:rate.f} 
I_l = \Lambda^*, \ I_u = \Lambda^\sharp & \text{under the assumption } S1
\nonumber \\
I_l=I_u = \Lambda^* & \text{under the assumption } S2 \nonumber \\
I_l=I_u = \bigl( G_\Sigma\bigr)^* & \text{under the assumption  } S3
\nonumber \\ 
I_l=I_u = \bigl( \Lambda^h\bigr)^* & \text{under the assumption } S4\\ 
I_l = \bigl(\Lambda_\alpha\bigr)^*, \ I_u =
\Lambda_\alpha^\sharp & \text{under the assumption } R1
\nonumber \\
I_l=I_u = \bigl(\Lambda_\alpha\bigr)^* & \text{under the assumption } R2
\nonumber \\
I_l=I_u = \bigl((G_\Sigma)_\alpha\bigr)^* & \text{under the assumption } R3 
\nonumber \\
I_l=I_u = \bigl((\Lambda^h)_\alpha\bigr)^* & \text{under the assumption } R4 
\nonumber 
\end{eqnarray}

Here, for a convex 
function $f:\, \reals^d\to (-\infty,\infty]$, we denote by $f^*$ 
its Legendre transform $f^*(x) = \sup_{\lambda\in\reals^d}\bigl\{
\lambda\cdot x - f(\lambda)\bigr\}$, $x\in\reals^d$. Further, under
the assumption $S1$, $\Lambda ^\sharp(x) = \sup_{\lambda\in\Pi}\bigl\{
\lambda\cdot x - \Lambda (\lambda)\bigr\}$, with 
\begin{equation}\label{eq:pi}
\Pi = \bigl\{ \lambda\in \reals^d:\ \text{for some} \ N_\lambda, \ 
\sup_{n\geq N_\lambda, \, i\in\bbz}
\Lambda(\lambda\phi_{i,n})<\infty\bigr\}, 
\end{equation}
where $\phi_{i,n}=\phi_{i+1}+\cdots +\phi_{i+n}$, is a partial sum of the
moving average coefficients. Further, under the assumptions $S3$ and $R3$,
$G_\Sigma$ is the log-moment generating function of a zero mean
Gaussian random vector in $\reals^d$ with the same variance-covariance
matrix as that of $Z_0$. Next, under the assumptions $S4$ and $R4$,
$\Lambda^h(\lambda)=
\zeta_\Lambda(\lambda/\|\lambda\|)\|\lambda\|^\beta$. Under the
assumptions $R1$-$R4$, for a nonnegative measurable function $f$ on
$\reals^d$ we define 
\begin{equation}\label{eq:falpha}
f_\alpha(\lambda) = 
\int_{-\infty}^\infty f\left(\lambda
(1-\alpha)\int_x^{x+1}|y|^{-\alpha}\big( p\one(y\geq
0)+q\one(y<0)\bigr)\, dy\right)\, dx 
\end{equation}
if $1/2<\alpha<1$ and $f_1=f$. Finally, under the assumption $R1$, we
define 
$\Lambda_\alpha^\sharp(x) = \sup_{\lambda\in\Pi_\alpha}\bigl\{
\lambda~\cdot~x - \Lambda_\alpha(\lambda)\bigr\}$, with 
$\Pi_\alpha$ given by 
\begin{equation}\label{pir}
\Pi_{\alpha} \defn
\Big\{\lambda: (p\wedge
q)\lambda \in \mathcal{F}_\Lambda^\circ, \text{ and for   some }
N_\lambda,   \ \sup_{n\ge
  N_\lambda ,\, i\in\mathbb{Z}}
\Lambda\Big(\frac{\lambda \phi_{i,n}}{\Psi_n}\Big)<\infty \Big\}  
\end{equation} 
for $1/2<\alpha<1$, while for $\alpha=1$, we define 
\begin{equation}\label{pir.1}
\Pi_{1} \defn
\Big\{\lambda: \lambda \in \mathcal{F}_\Lambda^\circ, \text{ and for
  some } N_\lambda,   \ \sup_{n\ge
  N_\lambda ,\, i\in\mathbb{Z}}
\Lambda\Big(\frac{\lambda \phi_{i,n}}{\Psi_n}\Big)<\infty \Big\}  .
\end{equation} 

We are now ready to state the main result of this section. The
following theorem considers the various cases in 
Assumptions \ref{ass:short} and \ref{ass:long} 
and gives us the rate of growth of the lengths of the long strange
segments in each of the cases. For a set $A$ in $\reals^d$ and
$\eta>0$ we denote  
\begin{equation}\label{Aeta}
A(\eta) \defn\big\{x:d(x,A^c)>\eta\big\},
\end{equation} 
where $d(x,A^c)$ is the distance from the point $x$ to the complement 
$A^c$. 

\begin{theorem}\label{lss:short}
If any one of $S1$-$S4$ or $R1$-$R4$ hold, then for any Borel set
$A\subset\mathbb{R}^d$, 
\begin{equation} \label{e:Tlss.short}
I_*\le
\liminf_{r\rightarrow\infty}\frac{\log
T_r(A;\underline{a})}{b_r}\le\limsup_{r\rightarrow\infty}\frac{\log
T_r(A;\underline{a})}{b_r}\le I^* 
\end{equation}
and 
\begin{equation} \label{e:Rlss.short}
\frac{1}{I^\ast}\leq \liminf_{m\rightarrow\infty}\frac{b_{R_m}}{\log m}
\leq \limsup_{m\rightarrow\infty}\frac{b_{R_m}}{\log m}
\leq \frac{1}{I_\ast}
\end{equation}
with probability 1, where, under the assumptions $S2, S3, S4, R2, R3$
and $R4$, 
\[
I_*= \inf\limits_{x\in \bar{A}}I_u(x) \ \ \mbox{ and } \ \ 
I^*= \inf\limits_{x\in A^\circ}I_l(x),
\]
with $I_l$ and $I_u$ as in \eqref{e:rate.f}. Under the assumption $S1$,
$I_*$ is defined in the same way, while $I^*$ is defined now as
follows. Let $\lambda^*=\sup\{ \lambda:\, \lambda\in \Pi\}>0$. Then 
$$
I^*= \inf_{\eta\in \Theta}\inf\limits_{x\in A(\eta)}I_l(x),
$$
where $\Theta = \{ \eta>0:\, \eta>(\lambda^*)^{-1} \inf\limits_{x\in
A(\eta)}I_l(x)\}$.  Finally, under the assumption $R1$, $I_*$ is defined
in the same way, and with 
$\lambda^*_\alpha=\sup\{ \lambda:\, \lambda\in \Pi_\alpha\}>0 $, and 
 $\Theta_\alpha = \{ \eta>0:\, \eta>(\lambda^*_\alpha)^{-1}
\inf\limits_{x\in A(\eta)}I_l(x)\}$, one sets 
$$
I^*= \inf_{\eta\in \Theta_\alpha}\inf\limits_{x\in A(\eta)}I_l(x). 
$$

\end{theorem}

\begin{remark} \label{rk:same.rates}
{\rm 
In certain cases it turns out that $I_*=I^*$ in Theorem
  \ref{lss:short}, and then its conclusions may be strengthened. 
For example, under the assumptions S2, S3, S4, R2, R3
or R4,  suppose that for some Borel set $A$, 
$$
\inf\limits_{x\in   A^\circ}I_l(x)=\inf\limits_{x\in 
    \bar{A}}I_u(x)=I \ \text{(say)}.
$$
Then, with probability 1,  
\begin{equation}\label{onelimit} 
\lim_{r\rightarrow \infty} \frac{\log T_r}{b_r}=I 
\end{equation}
and
\begin{equation}\label{onelimit.R}
\lim_{m\rightarrow\infty}\frac{b_{R_m}}{\log m} = \frac{1}{I}\,.
\end{equation}
}
\end{remark} 

Because of the large deviation principle for the sequence $(\mu_n)$, 
 the sequence $(b_n)$ is the ``right''
  normalization to use in the Theorem \ref{lss:short}. In particular,
 if, for 
  instance, the set $A$ is bounded away from the origin (which we
  recall to be the mean of the moving average process), then the
  quantity $I_*$ is strictly positive. Under further additional
  assumptions on the set $A$ the quantity $I^*$ will be finite, and
  then \eqref{e:Tlss.short} and \eqref{e:Rlss.short} give us precise
  information on the order of magnitude of long strange segments. 

Notice that under the ``usual'' normalization  $a_n=n$, Theorem
\ref{lss:short} says that $R_m$ grows like $\log m$ in the short
memory case (i.e. under the assumption $S1$); see also  Theorem 3.2.1 in
\cite{dembo:zeitouni:1998}. On the other hand, in the long memory
case, it is easy to see that the case $a_n=n$ falls into 
 the assumption $R3$, and then  the length $R_m$ of the long strange
segments grows at the rate $\Theta(\log m)$, where $\Theta$ is regularly
varying at infinity with exponent $1/(2\alpha-1)$. Therefore, long
strange segments are much longer in the long memory case than in the short
memory case. In fact, to get long strange segments with length of
order $\log m$ in the long memory case one needs to use a stronger
normalization $a_n=n\Psi_n$ (the assumptions $R1$ and $R2$). This phase
transition property is directly inherited from the similar phenomenon
for large deviations; see \cite{ghosh:samorodnitsky:2009}. 

To emphasize more generally the difference between the length of the 
long strange segments in the two cases we summarize in the table below
the 
corresponding statements of Theorem \ref{lss:short} for $(a_n)$ being
a regularly varying sequence with exponent $\omega\geq 1/2$ of regular
variation. We will implicitly assume that the appropriate assumptions
of the theorem hold in each case, and that the limits $I_*$ and $I^*$
are positive and finite. The general statement is that, with
probability 1, $R_m$ is of the order $\Theta(\log m)$, where $\Theta $ is
regularly varying at infinity with some exponent $\theta$. We describe
$\theta$ as a function of $\omega$ in all cases. The value
$\theta=\infty$ corresponds to $R_m$ growing faster than any power of
$\log m$.  In all cases the long
strange segments are much longer in the long memory case than in the short
memory case. Recall that $-\alpha$ is the exponent of regular
variation of the coefficients in Assumption \ref{regularvarying}, and
$\beta$ is the exponent of regular variation of $\Lambda$ in
assumptions $S4$ and $R4$. 
\begin{table}[h]\label{tab:tab1}
\caption{The effect of memory on the rate of growth of Long Strange
  Segments of a Moving Average Process} 
\begin{center}
\begin{tabular}{|c|c|c|c|}
\hline
Range of $\omega$ &Assumptions& Short memory & Long memory\\ 
\hline
$\frac12\leq \omega\leq \frac32 - \alpha$ &$S3$ &$ \theta =
\frac{1}{2\omega-1} $& $\theta=\infty $\\ 
$\frac32 - \alpha \leq \omega\leq 1$ &$S1$, $S2$, $S3$, $R3$& $\theta = \frac{1}{2\omega -1}$& $\theta=\frac{1}{2\omega +2\alpha -3} $\\ 
$1\leq \omega \leq 2-\alpha $ &$S4$, $R1$,
$R2$, $R3$ & $\theta = \frac{\beta-1}{\beta \omega
  -1}$& $\theta=\frac{1}{2\omega +2\alpha -3}$\\  
$\omega \geq 2-\alpha$ &$S4$, $R4$&$ \theta = \frac{\beta-1}{\beta \omega -1}$
& $\theta=\frac{\beta-1}{\beta(\omega +\alpha-1)-1} $\\ 
\hline
\end{tabular}
\end{center}
\end{table}
Notice that the long range dependent case in the first row of the
table does not correspond to any assumption we have made. The fact
that $\theta=\infty$ in this case follows as one of the extreme cases
of the second row in the table.

\begin{proof}[Proof of Theorem \ref{lss:short}]
The duality relation $\{ R_m(A;\underline{a})\geq r\} = \{
T_r(A;\underline{a})\leq m\}$ and monotonicity of the sequence $(b_n)$
imply that the statements \eqref{e:Tlss.short} and
\eqref{e:Rlss.short} are equivalent. We will, therefore, concentrate
on proving \eqref{e:Tlss.short}. The proof of the lower bound is
standard, and does not rely on the fact that the underlying process is
a moving average; see Theorem 3.2.1 in \cite{dembo:zeitouni:1998}. We
include an argument for completeness. Note that for every $r,m\geq 1$
$$
P\bigl( T_r(A;\underline{a})\leq m\bigr) \leq m\sum_{n=r}^\infty
\mu_n(A;\underline{a})\,.
$$
If $I_*=0$, there is nothing to prove. Suppose that
$0<I_*<\infty$. Choose $0<\vep<I_*$. By the definition of $I_*$ and
the large deviation principle \eqref{e:LDP}, we
know that there is $c=c_\vep\in (0,\infty)$ such that
$\mu_n(A;\underline{a}) \leq ce^{-b_n(I_*-\vep/2)}$ for all $n\geq
1$. Choosing $m=\lfloor e^{b_r(I_*-\vep)}\rfloor$ gives us
	\begin{eqnarray}
	\sum_{r=1}^\infty P(T_r\le e^{b_r(I_*-\vep)})&\le &
	\sum_{r=1}^\infty e^{b_r(I_*-\vep)}\sum_{n=r}^\infty
	ce^{-b_n(I_*-\vep/2)} \nonumber\\ 
	& \le & c^\prime \sum_{r=1}^\infty e^{-b_r\vep/2}<\infty
	\nonumber 
	\end{eqnarray}
for some positive constant $c^\prime$ (depending on $\vep$). Using the
first Borel-Cantelli lemma and letting $\vep\downarrow 0$ established
the lower bound in \eqref{e:Tlss.short}. When $I_*=\infty$, we take
any $\vep>0$ and observe that by the definition of $I_*$ 
there is $c=c_\vep\in (0,\infty)$ such that
$\mu_n(A;\underline{a}) \leq ce^{-2b_n/\vep}$ for all $n\geq
1$. Choose now $m=\lfloor e^{b_r/\vep}\rfloor$ and proceed as above to
conclude that 
 \[
\sum_{r=1}^\infty P(T_r\le e^{b_r/\vep})<\infty,
\]
after which one uses, once again, the
first Borel-Cantelli lemma and lets $\vep\downarrow 0$ to obtain the
lower bound in \eqref{e:Tlss.short}.  

For the upper bound in \eqref{e:Tlss.short}, we only need to consider 
the case $I^*<\infty$. In that case the set $A$ has nonempty interior.
Define two new probability measures by  
$$
	\mu_n^\prime(\cdot)\defn P\Big(\frac{1}{a_n}\sum_{|i|\le n^2}
	\phi_{i,n}Z_i\in \cdot\Big)\mbox{ and
	}\mu_n^{\prime\prime}(\cdot)\defn
	P\Big(\frac{1}{a_n}\sum_{|i|> n^2} \phi_{i,n}Z_i\in
	\cdot\Big), 
$$
where, as before, $\phi_{i,n}=\phi_{i+1}+\cdots +\phi_{i+n}.$

For  any sequence  $(k_n)$ of integers, with $k_n/n\rightarrow
\infty$,  and any $\lambda>0$ under the assumptions $S2, S3, S4, R2, R3$
and $R4$, any $\lambda\in \Pi$ under the assumption $S1$, or any
$\lambda\in \Pi_\alpha$ under the assumption $R1$,  
\begin{equation} \label{e:same.limit}
	\lim_{n\rightarrow \infty}
	\frac{1}{b_n}\sum_{i=-k_n}^{k_n}\Lambda \Big( \frac{b_n}{a_n}
	\lambda\phi_{i,n} \Big) = \lim_{n\rightarrow \infty}
	\frac{1}{b_n}\sum_{i=-\infty}^{\infty}\Lambda
	\Big(\frac{b_n}{a_n} \lambda\phi_{i,n} \Big); 
\end{equation}
see Remark 3.7 in \cite{ghosh:samorodnitsky:2009}. This means that the
sequence $(\mu_n^\prime)$ satisfies the LDP with  speed $b_n$ and
same upper rate functions $I_u$ given in \eqref{e:rate.f} 
as the sequence $(\mu_n)$. The fact that
the same is true for the lower rate functions in \eqref{e:rate.f}  
follows from the
argument in theorems 2.2 and 2.4 in \cite{ghosh:samorodnitsky:2009}. 

For fixed integers $r,q$, and  $l=1,\ldots,\lfloor q/(2r^2+1) \rfloor$,
define
\[ 
B_l\defn \frac{1}{a_r}\sum_{i=1+(l-1)(2r^2+1)}^{r+(l-1)
  (2r^2+1)}X_i,
\] 
and 
\[ 
B_l^\prime\defn \frac{1}{a_r}
\sum_{j=-r^2}^{r^2}\phi_{j,r}Z_{-j+(l-1)(2r^2+1)}.
\]   
Since the $B_l^\prime$ are independent,  for any $r$ and  $q$ we have,
\begin{eqnarray*}
	&& P\big[T_r>q\big]\\
	 & \le & P\left[  B_l\notin A, l=1,\ldots, \left\lfloor
	 \frac{q}{2r^2+1}\right\rfloor  \right] \\ 
& \le & P\left[  B^\prime_l\notin A(\eta),l=1,\ldots, \left\lfloor
	 \frac{q}{2r^2+1}\right\rfloor \right] 
     + \sum_{l=1}^{\lfloor q/(2r^2+1) \rfloor} P\Big[
	 |B_l-B_l^\prime|>\eta\Big]\\ 
&= & \Big( 1-\mu_r^\prime\big(A(\eta) \big) \Big)^{\lfloor q/(2r^2+1)
	 \rfloor} + \sum_{l=1}^{\lfloor q/(2r^2+1) \rfloor} P\Big[
	 |B_l-B_l^\prime|>\eta\Big]\\ 
& \le & \exp\Big(-\frac{q}{2r^2+1}\mu_r^\prime\big(A(\eta)\big) \Big)+
	 \frac{q}{2r^2+1}\mu_r^{\prime\prime}\big(\{x:|x|>\eta\}\big). 
\end{eqnarray*}
By the definition of $I^*$ and
the large deviation principle \eqref{e:LDP}, for any $\vep>0$ 
there is $c=c_{\vep}\in (0,\infty)$ such that for all $\eta>0$ small
enough, 
$\mu_n^{ \prime }(A(\eta)) \geq ce^{-b_n(I^*+\vep/2)}$ for all $n$ large than
some $n_{\vep}$. Therefore, fixing $\vep>0$ and using the bound
above with  $q= e^{b_r(I^*+\vep)}$, we see that for some
$C=C_{\vep}\in (0,\infty)$, for all $\eta>0$ small enough, 
\begin{eqnarray}\label{lss:pt1}
\sum_{r=1}^\infty\exp\Big(-\frac{ e^{b_r(I^*+\epsilon)}
}{2r^2+1}\mu_r^\prime\big(A(\eta)\big) \Big)&\le
&C\sum_{r=1}^\infty\exp\Big(-c\frac{ e^{b_r(I^*+\epsilon)}
}{2r^2+1}e^{-b_r(I^*+\epsilon/2)} \Big) \nonumber \\ 
& = & C\sum_{r=1}^\infty \exp\Big(-
c\frac{e^{b_r(\epsilon/2)}}{2r^2+1}\Big)<\infty. 
\end{eqnarray}
Suppose first that we are under the assumptions S2, S3, S4, R2, R3
or R4. Fixing $\vep>0$ and choosing $\eta>0$ small enough for the above
to hold, we see that 
\begin{eqnarray*}
	&&  \limsup_{ n \rightarrow \infty} \frac{1}{b_n} \log
	\mu_n^{\prime\prime}\big(\{x:|x|>\eta\} \big) \\ 
	& \le & \limsup_{ n \rightarrow \infty} \frac{1}{b_n} \log
	\left( e^{-b_n\lambda \eta} E\Big[\exp \Big\{ \lambda
	\frac{b_n}{a_n} \sum_{|i|>n^2}\phi_{i,n}Z_i\Big\} \Big]
	\right)\\ 
	& = & -\lambda \eta + \limsup_{ n \rightarrow
	\infty}\frac{1}{b_n}\sum_{|i|>n^2}\Lambda\Big(\frac{b_n}{a_n}\lambda
	\phi_{i,n} \Big) = -\lambda\eta, 
\end{eqnarray*}
with the last equality following from \eqref{e:same.limit}. 
Choosing now $\lambda>(I^*+\epsilon)/\eta$ (which is possible under
the current assumptions no matter how small $\eta>0$ is), we obtain 
\begin{equation}\label{lss:pt2}
 \sum_{r=1}^\infty
 \frac{e^{b_r(I^*+\epsilon)}}{2r^2+1}
\mu_r^{\prime\prime}\big(\{x:|x|>\eta\}\big) <\infty. 
\end{equation} 
Combining (\ref{lss:pt1}) and (\ref{lss:pt2}) we have
$\sum_{r=1}^\infty P\Big[ T_r>e^{b_r(I^*+\epsilon)}\Big]<\infty$, so
that using the first Borel-Cantelli lemma gives and letting
$\vep\downarrow 0$ proves the upper bound in \eqref{e:Tlss.short}. The
cases of the assumptions $S1$ and $R1$ are the same, except now $\lambda$
cannot be taken to be arbitrarily large, which restricts the feasible
values of $\eta>0$. This completes the proof.  
\end{proof}

\end{section}

\begin{section}{Ruin Probabilities}\label{sec:ruin}

This section discusses the rate of decay ruin probability for a moving
average  process $ (   X_{ n},n\in \mathbb{Z}  )$ in \eqref{maproc}. 
We study the  probability of ruin in infinite time, defined as 
\begin{equation}\label{ruindefn}
	\rho(u; A;\underline{a};\mu)= \rho(u)=P \left[ Y_n\in
          uA\ \text{ for some }n\ge 1  \right] 
\end{equation}
where $(Y_n)$ is given by \eqref{eq:ruin} for some $\mu\in \bbr^d$ and
a sequence $ \underline{a}=(a_{ n})$ increasing to $ \infty$, and $A\subset\bbr^d$ is a
Borel set. A related notion is the {\it time of ruin} defined by
\[
	T(u)=\inf \left\{ n:\, Y_n\in  uA  \right\}.
\]
Clearly, $\rho(u)=P[T(u)<\infty]$. We will study the asymptotic
behavior of $ \rho(u)$ 
as  $ u$ increases to infinity.   

Our main results are in the following theorems, roughly
corresponding to assumptions \ref{ass:short} and \ref{ass:long} of the
previous section. We start with the short memory regimes. 

\begin{theorem}\label{thm:ma:rp:sm} 
If $ S1$ holds, then 
$$
-\inf_{t\in {\mathcal F}}r(t)\bigl[ t\nabla \Lambda(t) -
  \Lambda(t)\bigr]
\leq \liminf_{u\rightarrow \infty}\frac{1}{u}\log \rho(u)
$$
$$
\leq 
\limsup_{u\rightarrow \infty}\frac{1}{u}\log \rho(u)\leq -
\sup_{t\in {\mathcal D}}\inf_{\gamma\in A}t\gamma\,,
$$
where 
$$
{\mathcal D} = \Bigl\{ t\in \bbr^d:\, \inf_{\gamma\in
  A}t\gamma>  0, \, \sup_{n\geq   1} \Bigl[
  \sum_{i\in\mathbb{Z}}\Lambda\left(t\phi_{i,n}\right)-nt \mu
  \Bigr]<\infty\Bigr\},
$$
$$
{\mathcal F} = \bigl\{ t\in
     \Pi^\circ:\,  r\, \bigl(\nabla \Lambda(t)-\mu\bigr)\in
     A^\circ\ \text{for some $\rho>0$}\bigr\},
$$
and $r(t) = \inf\{r>0:\, r\, \bigl(\nabla
\Lambda(t)-\mu\bigr)\in A^\circ\}$.  
\end{theorem}

\begin{remark}\label{rem:explicit.sL}
{\rm
In certain cases Theorem \ref{thm:ma:rp:sm} provides a precise and
explicit statement. Suppose for simplicity that $\Lambda(t)<\infty$
for all $t$, and that the random variable $\mu Z$ is unbounded. Then
there exists a unique $w>0$ such that 
$$
\Lambda(w\mu) = w\|\mu\|^2.
$$
Assume that $r\bigl( \nabla \Lambda(w\mu)-\mu\bigr)\in A^\circ$ for
some $r>0$, and let 
$$
\gamma_* = r(w\mu)\bigl( \nabla \Lambda(w\mu)-\mu\bigr)\in
\overline{\bigl( A^\circ\bigr)}. 
$$
Then the lower bound in Theorem \ref{thm:ma:rp:sm} gives us  
$$
\liminf_{u\rightarrow \infty}\frac{1}{u}\log \rho(u) \geq
-w\gamma_*\mu\,.
$$
If we assume, additionally, that $\inf_{\gamma\in A}\mu\gamma>0$, then
it follows that $a\mu\in {\mathcal D}$ for any $0<a<w$, and a further
assumption $\gamma_*\in\argmin\bigl\{ \mu\gamma:\, \gamma\in A\bigr\}$
will allow us to conclude from the upper bound  Theorem
\ref{thm:ma:rp:sm} that
$$
\limsup_{u\rightarrow \infty}\frac{1}{u}\log \rho(u) \leq
-w\gamma_*\mu\,.
$$
Therefore, 
\begin{equation} \label{e:explicit.sL}
\lim_{u\rightarrow \infty}\frac{1}{u}\log \rho(u) = 
-w\gamma_*\mu\,.
\end{equation}
All of the assumptions are easily seen to be satisfied in the
one-dimensional case with $\mu>0$ and $A=(1,\infty)$. 
}
\end{remark}

For the next two theorems we introduce the following condition on the
set $A$.
\begin{condition}\label{cond:A} 
We say that a set $A\in\bbr^d$ satisfies Condition $\mathcal A$ if 
\begin{itemize}
\item there is $t\in\bbr^d$ such that $t\mu>0$ and 
 $\inf_{\gamma\in A} t\gamma>0$; 
\item for any $x\in A$ and $\rho>0$, $x+\rho\mu\in A$ and
  $(1+\rho)x\in A$.  
\end{itemize}
\end{condition}

\begin{theorem}\label{thm.sm.1}
Suppose that the set $A$
  satisfies Condition $\mathcal A$ (Condition \ref{cond:A}). If  $S3$
  holds, and $ (a_{ n})\in RV_{ \omega }$ for  some   $1/2<\omega\leq
  1$, then 
$$
-\inf_{c>0} \left[ c^{-(2w-1)/w}\inf_{\gamma\in A^\circ}\left( \frac12(\mu
  +c\gamma)^\prime\Sigma^{-1}(\mu   +c\gamma)\right)\right]
\leq \liminf_{ u \rightarrow \infty  }   \frac{ 1}{b_{ a^{ \leftarrow }(u)} } \log \rho(u)
$$
$$
\leq \limsup_{ u \rightarrow \infty  }   \frac{ 1}{b_{ a^{ \leftarrow }(u)} } \log \rho(u)
\leq -\inf_{c>0} \left[ c^{-(2w-1)/w}\inf_{\gamma\in A}\left(
  \frac12(\mu 
  +c\gamma)^\prime\Sigma^{-1}(\mu   +c\gamma)\right)\right],
$$
where the inverse of $(a_n)$ is defined by $a^{\leftarrow
}(u) = \inf\{n\geq 1:\, a_n\geq u\}$, $u>0$.  
\end{theorem}

\begin{remark}\label{rem:explicit.sG}
{\rm
Again, in certain cases the statement of Theorem \ref{thm.sm.1} takes 
a very explicit form. Suppose, for example, that 
\begin{equation} \label{e:gamma.0}
\text{there is} \ \gamma_0\in \overline{ \bigl( A^\circ\bigr)}
\ \ \text{such that}
\end{equation}
$$
\gamma^\prime\, \Sigma^{-1}\gamma \geq \gamma_0^\prime\, \Sigma^{-1}\gamma_0 
\ \ \text{and}
\ \ \mu^\prime\, \Sigma^{-1}\bigl(\gamma-\gamma_0\bigr)\geq 0
\ \ \text{for all $\gamma\in A$.}
$$
This would be, for instance, the situation in the one-dimensional case
with $\mu>0$ and $A=(1,\infty)$. Under this assumption, for every
$c>0$, 
$$
\inf_{\gamma\in A}\left(
 (\mu 
  +c\gamma)^\prime\Sigma^{-1}(\mu   +c\gamma)\right)
= \inf_{\gamma\in A^\circ}\left( (\mu
  +c\gamma)^\prime\Sigma^{-1}(\mu   +c\gamma)\right)
$$
$$
= \left( (\mu
  +c\gamma_0)^\prime\Sigma^{-1}(\mu   +c\gamma_0)\right),
$$
and so optimizing over $c>0$ we obtain 
\begin{equation} \label{e:explicit.sG}
\lim_{ u \rightarrow \infty  }   \frac{ 1}{b_{ a^{ \leftarrow }(u)} } \log \rho(u)
= -\frac12 c_0^{-(2w-1)/w} \left( (\mu
  +c_0\gamma_0)^\prime\Sigma^{-1}(\mu   +c_0\gamma_0)\right)\,,
\end{equation}
where 
$$
c_0 = \frac{\sqrt{(2w-1)\bigl[ \bigl( \mu^\prime\, \Sigma^{-1}\mu\bigr)
\bigl( \gamma_0^\prime\, \Sigma^{-1}\gamma_0\bigr) 
- \bigl(\mu^\prime\, \Sigma^{-1}\gamma_0\bigr)^2\bigr] +w^2
    \bigl(\mu^\prime\,
    \Sigma^{-1}\gamma_0\bigr)^2}}{\bigl(\gamma_0^\prime\, 
  \Sigma^{-1}\gamma_0\bigr)}
$$
$$
-(1-w)\frac{\bigl(\mu^\prime\,
\Sigma^{-1}\gamma_0\bigr)}{\bigl(\gamma_0^\prime\,
  \Sigma^{-1}\gamma_0\bigr)}. 
$$
}
\end{remark}

\begin{theorem}\label{thm.sm.1a}
Suppose that the set $A$
  satisfies Condition $\mathcal A$ (Condition \ref{cond:A}). If  $S4$
  holds, and $ (a_{ n})\in RV_{ \omega }$  for
  some   $\omega\geq 1$, then 
\[ 
-\inf_{c>0} \left[ c^{-\nu/w}\inf_{\gamma\in A^\circ} \bigl(
  \Lambda^h\bigr)^*(\mu +c\gamma)\right]
\leq \liminf_{ u \rightarrow \infty  }   \frac{ 1}{b_{ a^{ \leftarrow
    }(u)} } \log \rho(u) 
\]
$$
\leq \limsup_{ u \rightarrow \infty  }   \frac{ 1}{b_{ a^{ \leftarrow
    }(u)} } \log \rho(u) \leq \inf_{c>0} \left[
  c^{-\nu/w}\inf_{\gamma\in {\bar A}} \bigl( 
  \Lambda^h\bigr)^*(\mu +c\gamma)\right],
$$
where 
$$
\nu = 1+ (\omega-1)\frac{\beta}{\beta-1}\,.
$$
\end{theorem}

\begin{remark}\label{rem:explicit.sH}
{\rm
Once again, in certain cases the statement of Theorem \ref{thm.sm.1a} takes 
a very explicit form. Let us suppose, for example, that 
\begin{equation} \label{e:gamma.0a}
\text{there is} \ \gamma_0\in \overline{\bigl( A^\circ\bigr)}  
\ \ \text{such that}
\end{equation}
$$
\|\gamma\| \geq \|\gamma_0\| 
\ \ \text{and}
\ \ \mu^\prime\, \bigl(\gamma-\gamma_0\bigr)\geq 0
\ \ \text{for all $\gamma\in \overline{A}$.}
$$
Suppose, further, that for some $a>0$ the function $\Lambda$ satisfies 
\begin{equation} \label{e:nice.Lambda}
\zeta_\Lambda(\lambda) = a \ \ \text{for any unit vector $\lambda$ such that}
\end{equation}
$$
\lambda\mu>0 \ \ \text{or $\lambda\gamma>0$ for some $\gamma\in A$.}
$$
Again, this would be the the situation in the one-dimensional case
with $\mu>0$ and $A=(1,\infty)$. Under the assumption
\eqref{e:nice.Lambda}, 
$$
\bigl(\Lambda^h\bigr)^*(\mu +c\gamma) =
K_\beta \|\mu +c\gamma\|^{\beta/(\beta-1)}
$$
for any $c>0$ and $\gamma\in \overline{A}$, with 
$$
K_\beta= (\beta-1)\bigl( a\beta^\beta\bigr)^{1/(1-\beta)}.
$$
This, together with the assumption \eqref{e:gamma.0a}, 
implies that, for any $c>0$, 
$$
\inf_{\gamma\in A^\circ} \bigl( \Lambda^h\bigr)^*(\mu +c\gamma)
= \inf_{\gamma\in {\bar A}} \bigl( 
  \Lambda^h\bigr)^*(\mu +c\gamma) = K_\beta \|\mu
  +c\gamma_0\|^{\beta/(\beta-1)}.
$$
Optimizing over $c>0$ we obtain
\begin{equation} \label{e:explicit.sH}
\lim_{ u \rightarrow \infty  }   \frac{ 1}{b_{ a^{ \leftarrow }(u)} } \log \rho(u)
= - K_\beta c_0^{-\nu/w} \|\mu   +c_0\gamma_0\|^{\beta/(\beta-1)}\,,
\end{equation} 
where
$$
c_0 = \frac{\sqrt{4(\beta w-1)\bigl(
   \| \mu\|^2\|\gamma_0\|^2-(\mu\gamma_0)^2\bigr) +\beta^2
    w^2(\mu\gamma_0)^2}+ (\beta w-2)\mu\gamma_0}{2\|\gamma_0\|^2}.
$$
}
\end{remark}

We now turn to the asymptotic behavior of the ruin probabilities in
the long memory regimes. In all 3 theorems we assume that the set $A$
satisfies Condition $\mathcal A$. Note in the following theorem $ b_{ n}=n$ and therefore $ b_{ a^{ \leftarrow }(u)}$ reduces to $ a^{ \leftarrow }(u)$.

\begin{theorem}\label{thm:ma:rp:lm} 
Suppose that the set $A$
  satisfies Condition $\mathcal A$ (Condition \ref{cond:A}). If $R2$
  holds, then  
\[
-\inf_{ c>0}c^{ \frac{ 1}{\alpha -2 } }\inf_{\gamma\in A^\circ}
(\Lambda _{ \alpha })^{ *}\big(\mu+c\gamma\big)\leq 
	\liminf_{ u \rightarrow \infty   } \frac{ 1}{ a^{ \leftarrow
          }(u)} \log \rho( u)  
\]
\begin{eqnarray*}
	&\leq & \limsup_{ u \rightarrow \infty   } \frac{1}{a^{
      \leftarrow }(u)}\log \rho( u)\\ 
	 & \le&\left\{\begin{array}{ll} \inf\limits_{ t\in G}
  \sup\limits_{ u> 0}  \left\{-u^{ \alpha -1}\inf_{\gamma\in A} t\gamma
 +u\Big( \Lambda _{\alpha }(t)-t\mu \Big)  \right\} & \mbox{if }\alpha
 <1 \\ -\inf_{c>0}c^{ -1 }\inf_{\gamma\in \bar A} 
\Lambda^{ *}\big(\mu+c\gamma\big)&
  \mbox{if } \alpha =1.\end{array} \right. 
\end{eqnarray*}
where 
$$
G=\{  t\in  \bbr^d:\, t\mu>0, \ \inf_{\gamma\in A} t\gamma>0
\ \ \text{and} \ \ \Lambda _{ \alpha }(t)-\mu t<0\},
$$ 
and $\Lambda_{ \alpha }(\cdot) $ is defined in \eqref{eq:falpha}. 
\end{theorem}
Observe that the set $G$ in the above theorem is not empty because of
Condition $\mathcal A$ and the fact that $|\Lambda _{ \alpha }(t)|\leq
c|t|^2$ for $t$ in a neighborhood of the origin. 

To state the next two theorems we introduce the notation
\begin{equation} \label{e:C}
 	C_{ \alpha ,\beta }= \left\{ \begin{array}{ll}    (1-\alpha
          )^{ \beta }   \int\limits_{ - \infty }^{ \infty } \Big(
          \int\limits_{ x}^{ x+1} \big| y \big|^{ -\alpha }\big( pI_{[
              y\ge 0]}+qI_{[  y<0]} \big)dy \Big)^{ \beta } dx &
          \mbox{if } \alpha <1\\ 1 & \mbox{if } \alpha
          =1\\ \end{array} \right.
\end{equation}
for $ 1/2<\alpha \le 1$ and $ \beta >1$. 

\begin{theorem}\label{thm:ma:rp:lm.1}
Suppose that the set $A$
  satisfies Condition $\mathcal A$ (Condition \ref{cond:A}). If $R3$
  holds, and $ (a_{ n})\in RV_{ \omega }$ for  some $3/2-\alpha<
  \omega\leq 2-\alpha$, then 
$$
-\frac{1}{C_{ \alpha ,2}}\inf_{c>0} \left[
  c^{-2+(3-2\alpha)/w}\inf_{\gamma\in A^\circ}\left( \frac12(\mu 
  +c\gamma)^\prime\Sigma^{-1}(\mu   +c\gamma)\right)\right]
\leq \liminf_{ u \rightarrow \infty  }   \frac{ 1}{b_{ a^{ \leftarrow
    }(u)} } \log \rho(u) 
$$
\begin{eqnarray*}
&\leq & \limsup_{ u \rightarrow \infty  }   \frac{ 1}{b_{ a^{ \leftarrow
    }(u)} } \log \rho(u) \\
&\leq& \left\{\begin{array}{ll} -K_{\alpha,\omega} \sup\limits_{ t\in G}
 \Bigl[\bigl( t\mu -\frac12
  C_{\alpha,2}\, t^\prime \Sigma t\bigr)^{-1+(3-2\alpha)/\omega}
\bigl(  \inf_{\gamma\in   A}  t\gamma\bigr)^{2-(3-2\alpha)/\omega}\Bigr] 
& \mbox{if }\alpha <1 \\
-\frac{1}{C_{ \alpha ,2}}\inf_{c>0} \left[
  c^{-2+(3-2\alpha)/w}\inf_{\gamma\in A}\left( \frac12(\mu 
  +c\gamma)^\prime\Sigma^{-1}(\mu   +c\gamma)\right)\right]&
  \mbox{if } \alpha =1\end{array}\right.,
\end{eqnarray*}
where
$$
G=\{  t\in  \bbr^d:\, t\mu>0, \ \inf_{\gamma\in A} t\gamma>0
\ \ \text{and} \ \ \frac12 C_{\alpha,2}\, t^\prime\Sigma t-\mu t<0\},
$$ 
and
$$
K_{\alpha,\omega}= \frac{w\bigl(
  3-2\alpha-\omega\bigr)^{1-(3-2\alpha)/\omega}}{\bigl(
  2(\alpha+\omega)-3\bigr)^{2-(3-2\alpha)/\omega}}. 
$$
\end{theorem}

\begin{remark}\label{rem:explicit.rG}
{\rm 
It is easy to check that in the one-dimensional case
with $\mu>0$, $A=(1,\infty)$ and $\Sigma = \sigma^2$, the statement
of the theorem gives the explicit limit
\[
	 \lim_{ u \rightarrow \infty  } \frac{ 1}{b_{ a^{ \leftarrow
             }(u)} }\log \rho(u) = - \frac{ \big( 2(\omega +\alpha )-3
           \big)^{ \frac{ 3-2(\omega +\alpha )}{\omega  }} }{ \big(
           3-2\alpha  \big)^{  \frac{ 3-2\alpha }{\omega  }}  } \frac{
           2}{\sigma ^{ 2}C_{ \alpha ,2} }\omega^{ 2} \mu^{ \frac{
             3-2\alpha }{ \omega }} . 
\]
One can check that under certain assumptions similar explicit
expressions can be obtained in the multivariate case as well. 
}
\end{remark}

\begin{theorem}\label{thm:ma:rp:lm.1a}
Suppose that the set $A$
  satisfies Condition $\mathcal A$ (Condition \ref{cond:A}). If $R4$
  holds, and $ (a_{ n})\in RV_{ \omega }$ for  some $\omega\geq
  2-\alpha$, $\omega\not= \beta(1-\alpha)+1$, then  
$$
-\inf_{c>0} \left[
  c^{-(\beta(\omega+\alpha-1)-1)/\omega(\beta-1)}\frac{\inf_{\gamma\in
    A^\circ} \bigl(   \Lambda^h\bigr)^*(\mu +c\gamma)}{\bigl(C_{ \alpha
    ,\beta}\bigr)^{1/(\beta-1)}}\right]
$$
$$
\leq \liminf_{ u \rightarrow \infty  }   \frac{ 1}{b_{ a^{ \leftarrow
    }(u)} } \log \rho(u) \leq \limsup_{ u \rightarrow \infty  }
\frac{ 1}{b_{ a^{ \leftarrow     }(u)} } \log \rho(u)
$$
$$
\leq \left\{\begin{array}{ll} -K^{(1)}_{\alpha,\beta,\omega} \sup\limits_{ t\in G^{(1)}} 
\left[\frac{\bigl( t\mu - 
  C_{\alpha,\beta}\,\Lambda^h(t)\bigr)^{(\beta(1-\alpha)+1-\omega)/\omega(\beta-1)}}{
\bigl(  \inf_{\gamma\in   A}
t\gamma\bigr)^{(1-\beta(\omega+\alpha-1))/\omega(\beta-1)}}\right]
& \mbox{if } \ \omega<\beta(1-\alpha)+1 \\
- \sup\limits_{ t\in G^{(2)}} \left[ \inf_{\gamma\in   A} t\gamma
- K^{(2)}_{\alpha,\beta,\omega}\frac{\bigl(
  C_{\alpha,\beta}\,\Lambda^h(t)\bigr)^{\omega/(\omega-1-\beta(1-\alpha))}}{(t\mu)^{(1+
\beta(1-\alpha))/(\omega-1-\beta(1-\alpha))}}\right] & \mbox{if }\ 
\omega>\beta(1-\alpha)+1 \end{array}\right. 
$$
if $\alpha<1$, and 
$$
\leq -\inf_{c>0} \left[
  c^{-(\beta\omega-1)/\omega(\beta-1)}
\inf_{\gamma\in
    \bar A} \bigl(   \Lambda^h\bigr)^*(\mu +c\gamma)\right]
$$
if $\alpha=1$. Here 
$$
G^{(1)}=\{  t\in  \bbr^d:\, t\mu>0, \ \inf_{\gamma\in A} t\gamma>0 
\ \ \text{and} \ \  C_{\alpha,\beta}\, \Lambda^h(t)-\mu t<0\},
$$  
$$
G^{(2)}=\left\{  t\in  \bbr^d:\, t\mu>0, \ \ \inf_{\gamma\in A} t\gamma>
K^{(2)}_{\alpha,\beta,\omega}\frac{\bigl(
  C_{\alpha,\beta}\,\Lambda^h(t)\bigr)^{\omega/(\omega-1-\beta(1-\alpha))}}{(t\mu)^{(1+
\beta(1-\alpha))/(\omega-1-\beta(1-\alpha))}}\right\}, 
$$
and 
$$
K^{(1)}_{\alpha,\beta,\omega}= \frac{\omega(\beta-1)\bigl(
\beta(1-\alpha)+1-\omega\bigr)^{-(\beta(1-\alpha)+1-\omega)/\omega(\beta-1)}}{\bigl( 
  \beta(\omega+\alpha-1)-1\bigr)^{(\beta(\omega+\alpha-1)-1)/\omega(\beta-1)} 
}, 
$$
$$
K^{(2)}_{\alpha,\beta,\omega}= \frac{\bigl(
  \omega-1-\beta(1-\alpha)\bigr)
\bigl( 1+\beta(1-\alpha)\bigr)^{(1+
\beta(1-\alpha))/(\omega-1-\beta(1-\alpha))}}{\omega^{\omega/(\omega
-1-\beta(1-\alpha))}}. 
$$
\end{theorem}

\begin{remark}\label{rem:explicit.rLL}
{\rm
Once again, the sets $G^{(1)}$ and $G^{(2)}$ in the theorem are not
empty. In the one-dimensional case
with $\mu>0$, $A=(1,\infty)$ and $\Lambda^h(t) = \xi_+ t^\beta$ for
$t>0$, the statement of the theorem gives the explicit limit
\begin{eqnarray*}
	&& \lim_{ u \rightarrow \infty  } \frac{ 1}{b_{ a^{ \leftarrow
      }(u)} }\log \rho(u) \\ 
	 &=& - \frac{ \big( \beta (\omega +\alpha -1)-1 \big)^{ \frac{
        1-\beta (\omega +\alpha -1)}{\omega (\beta -1) }}
  }{\big(1+\beta (1-\alpha )\big)^{ \frac{ 1+\beta (1-\alpha )}{\omega
        (\beta -1) }} } (\beta -1) \Big( \frac{ \omega^{ \beta
  }}{\xi_{ +}C_{\alpha , \beta } } \Big)^{ \frac{ 1}{\beta -1 }} \mu^{
    \frac{ 1+\beta (1-\alpha )}{\omega (\beta -1) }} . 
\end{eqnarray*}
}
\end{remark}

\begin{remark}\label{rem:rem:rp:longshort} 
{\rm 
As in the previous section, we clearly see how 
long range dependent variables $(X_n)$ (the ``claim sizes'') 
 influence the behavior of the ruin probability. Assume that the
 relevant upper bounds are finite and the relevant lower bounds are
 positive.  In the classical case of a linear sequence $ (a_{ n})$, in
 the short memory case (i.e. under the assumption $S1$), we have 
 \[ 
 	\log \rho(u)\approx  -c_{ S} u \ \ \ \  \mbox{ as } u\to
        \infty    
\]
for $ c_{ S} >0$, as in  Cram\'er's theorem. On the other hand,
in the long memory case the linear sequence falls into 
 the assumption $R3$, and then we have, instead,
\[ 
 	\log \rho (u) \approx -c_{ L} \frac{u}{ \Psi_{ u}^{2 }} 
        \ \ \ \ \mbox{ as } u\to \infty.    
\]
for $ c _{ L}>0$, and the right hand side above is in $RV_{
  2\alpha  -1}$, yielding a much larger ruin probability.

To further illustrate the effect of memory of a moving average process
on ruin probabilities we present Table 2, that presents the order of
magnitude of $-\log \rho (u)$ for large $u$ 
which we view in the form $-\log \rho (u) \approx -cu^\theta$ for
$c>0$. The tables presents dependence of $\theta$ on the exponent
$\omega$ of regular variation of the sequence $(a_n)$ in both short
and long memory cases. The value $\theta=0$ corresponds to the case
when $-\log \rho (u)$ grows slower than any positive power of
$u$. Notice that, for the same value of $\omega$, the value of
$\theta$ is always smaller in the long memory case than in the short
memory case, so that the ruin probability is much larger in the former
case than in the latter case. 

\begin{table}[h]\label{tab:tab2}
\caption{The effect of memory on the rate of decay of ruin probability when the claims process is a Moving Average.}
\begin{center}
\begin{tabular}{|c|c|c|}
\hline
Range of $\omega$ & Short range dependent & Long range dependent \\
\hline
$\frac12\leq \omega\leq \frac32 - \alpha$ &$ \theta = \frac{2\omega-1}{\omega } $& $\theta =0 $\\
$\frac32 - \alpha \leq \omega\leq 1$ & $\theta  = \frac{2\omega-1}{\omega }  $ & $\theta =\frac{2\omega +2\alpha -3}{\omega } $\\
$1< \omega < 2-\alpha $ & $\theta  = \frac{\beta \omega -1}{\omega(\beta-1)}$ & $\theta =\frac{2\omega +2\alpha -3}{\omega } $\\ 
$\omega \geq 2-\alpha$ &$ \theta  = \frac{\beta \omega -1}{\omega (\beta-1)}$ & $\theta =\frac{\beta(\omega +\alpha-1)-1}{\omega(\beta-1)} $\\
\hline
\end{tabular}
\end{center}
\end{table}
}
\end{remark}

\begin{proof}[Proof of Theorem \ref{thm:ma:rp:sm}.]
Notice  that for the moving   average process 
\[ 
 	   \log E\exp\Big(t\big(  S_{ n}-n\mu\big)\Big)=
           \sum_{i\in\mathbb{Z}}\Lambda\left(t\phi_{i,n}\right)-nt
           \mu. 
\]
The upper bound  follows immediately from part (i) of Theorem
\ref{t:nyrh3.1}. 

For the lower bound we apply part (ii) of Theorem \ref{t:nyrh3.1}. By
Lemma 3.5 (i) in  \cite{ghosh:samorodnitsky:2009}, $\Pi^\circ\subseteq
{\mathcal E}$, and for every $t\in \Pi^\circ$, $g(t) =
\Lambda(t)-t\mu$. The lower bound of part (i) of the present theorem 
follows. 
\end{proof}

\begin{proof}[Proof of Theorem \ref{thm.sm.1}.]
We start with the (easier) lower bound. We use the 
assumption of regular variation of $(a_n)$ as follows. First of all,
$b_n=a_n^2/n$ is regularly varying with exponent $2\omega-1$. Next,
for any $c>0$, 
$$
\frac{a^{ \leftarrow }(ca_{ n})}{nc^{ 1/\omega }}
= c^{-1/\omega } \frac{a^{ \leftarrow }(a_{ n})}{n}\frac{a^{ \leftarrow }(ca_{
    n})}{a^{ \leftarrow }(a_{ n})} 
\to 1
$$
as $n\to\infty$, see e.g. Theorem 1.5.12 in
\cite{bingham:goldie:teugels:1987}. Therefore, by the regular variation
of $(a_n)$ and $(b_n)$, 
\begin{eqnarray}\label{rp:lower:short}
	\liminf_{ u \rightarrow \infty   } \frac{ 1}{b_{ a^{ \leftarrow }(u)} }\log \rho( u)
	&= & \liminf_{n\rightarrow\infty} \frac{1}{b_{ a^{ \leftarrow }(ca_{ n})}}\log P\big[T( ca_{ n})<\infty \big] \nonumber \\
	&  \ge  & \liminf_{n\rightarrow\infty} \frac{ b_{ n}}{b_{ nc^{
              1/\omega }} } \frac{ 1}{b_{ n} }\log
        P\left[\frac{S_n}{a_{ n}}\in \mu+cA\right]  \\
	& \ge & -c^{ -(2\omega-1)/\omega}\inf_{\gamma\in
          A^\circ}\left( \frac12(\mu 
  +c\gamma)^\prime\Sigma^{-1}(\mu   +c\gamma)\right) \nonumber
\end{eqnarray}
by the large deviation principle; see \eqref{e:rate.f}. Now the lower
bound  follows by optimizing over $c>0$. 

Next we concentrate on the upper bound. We start with showing
that 
\begin{equation}\label{eq:al=1ptlast}
	\lim_{ M \rightarrow \infty  } \limsup_{ n \rightarrow \infty   } \frac{1}{b_{ n}} \log P \big[  nM<T( a_{ n})< \infty  \big] = - \infty. 
\end{equation}
To see this choose $t\in\bbr^d$ as in Condition $\mathcal A$ 
and $ \epsilon >0$ such that $
J(t)-t\mu +\epsilon <0$, where $J(t)=   \frac{1}{2}t\cdot\Sigma t$. 
For all $n$, 
\begin{eqnarray*}
 	P\big[  nM<T( a_{ n})< \infty  \big]  & = & \sum_{ k=nM+1}^{ \infty} P\big[ T(a_{ n})=k \big] \nonumber\\
	&\le &  \sum_{ k=nM+1}^{ \infty} P\big[ S_{ k}-a_{ k}\mu \in
          a_{ n}A\big] \\
        &\le &  \sum_{ k=nM+1}^{ \infty} P\big[ tS_{ k} -a_{ k}t\mu
          >a_n \inf_{\gamma\in A}t\gamma\big]. 
\end{eqnarray*}
Using Lemma 3.5(ii) in \cite{ghosh:samorodnitsky:2009} we know that
for all $n$ large enough, 
\[ 
 	\frac{ 1}{b_{ n} } \log E\Big[ \exp\Big(  t \frac{ b_{ n}}{a_{ n} } S_{ n} \Big)\Big]    \le J(t)+\epsilon .
\]
Therefore, applying an exponential Markov inequality we see that for
all $M$ large enough,  
\begin{eqnarray*}
	P\big[  nM<T( a_{ n})< \infty  \big] &\le &\sum_{ k=nM+1}^{
          \infty } \exp  \left\{ - \frac{ a_{ n}b_{ k}}{a_{ k}
        }\inf_{\gamma\in A}t\gamma  +  b_{ k} \Big( J(t)-\mu
        t+\epsilon \Big)  \right\}\\ 
	& \le & \sum_{ k=nM+1}^{  \infty } \exp  \left\{  b_{ k} \Big( J(t)-\mu t+\epsilon \Big)  \right\}\,.
\end{eqnarray*}
The assumption of regular variation of the sequence $(a_n)$ implies
that the sequence $(b_n)\in RV_{\nu}$ with 
$\nu = 2\omega-1$. Therefore, by 
Theorem 4.12.10 in \cite{bingham:goldie:teugels:1987} 
\[ 
 	\log    \sum_{ k=nM+1}^{  \infty } \exp  \left\{  b_{ k} \Big( J(t)-\mu t+\epsilon \Big)  \right\}\sim b_{ nM}\Big( J(t)-\mu t+\epsilon \Big)
\]
as $n\to\infty$, and so 
$$
	\limsup_{ n \rightarrow \infty   } \frac{ 1}{b_{ n} } \log P
        \big[  nM<T( a_{ n})< \infty  \big] \leq  M^{ \nu}\Big(
        J(t)-\mu t+\epsilon \Big)\,. 
$$
Now \eqref{eq:al=1ptlast} follows by letting $M\to\infty$. A similar
argument also shows that for any $ N\ge 1$, 
$$
	\lim_{ n \rightarrow \infty   } \frac{1}{b_{ n}} \log P \big[
          T( a_{ n})\le N \big] = - \infty, 
$$
and so  in order to prove the upper bound  of the theorem,
it suffices to show that  
\begin{eqnarray}\label{eq:al=1fin}
&&	\limsup_{ M \rightarrow  \infty } \limsup_{ N \rightarrow
          \infty   } \limsup_{ n \rightarrow \infty   }  \frac{1}{b_{
            n}}\log P \big[  N< T( a_{ n})\le nM \big] \\
&&\leq -\inf_{c>0} \left[ c^{-(2w-1)/w}\inf_{\gamma\in A}\left(
    \frac12\, (\mu 
  +c\gamma)^\prime\Sigma^{-1}(\mu   +c\gamma)\right)\right]. \nonumber
\end{eqnarray}
Notice that 
\begin{eqnarray*}
	 &  & P\big[  N<T( a_{ n})\le nM \big]\\
	 & = & P\big[ S_{ k}-a_{ k}\mu\in a_{ n}A \  \mbox{ for some }
    N<k\le nM\big] \\ 
	 & = & P\Bigg[ S_{ [ nMt]}\in a_{ nMt}\mu+a_{ n}A \  \mbox{
      for some } \frac{ N}{nM }<t\le 1 \Bigg] \\  
	 & = & P\Bigg[ Y_{ nM}( t)\in \frac{ a_{ nMt}}{a_{ nM} }\mu+
    \frac{ a_{ n}}{a_{ nM} }A \  \mbox{ for some } \frac{ N}{ nM}<t\le
    1\Bigg]. 
\end{eqnarray*}
Let $0<\delta<1$. By the Potter bounds, for all $ N\ge 1$ large enough
we have 
\[ 
 	\frac{ a_{ x}}{a_{ y} }>(1-\delta )\Big( \frac{ x}{y } \Big)^{
          \omega +\delta } \ \ \ \ \mbox{ for all }N<x\le y.     
\]
For such $ N$ and any $ n>N$ we have by the second part of Condition
$\mathcal A$,  
\begin{eqnarray*}
	 &  & P\big[  N<T( a_{ n})\le nM \big]\\
	 & \le & P \Bigg[  Y_{ nM}(t)\in (1-\delta )\Big( t^{\omega
      +\delta  }\mu + M^{ -(\omega +\delta)}A \Big) \  \mbox{ for some } \frac{ N}{nM }<t\le 1\Bigg]\\
	 & \le & P \Bigg[  Y_{ nM}(t)\in (1-\delta )\Big( t^{\omega
      +\delta  }\mu + M^{ -(\omega +\delta)}A \Big) \  \mbox{ for some
    } 0\leq t\le 1\Bigg]\\
	 & = & P\big[ Y_{ nM}\in B \big]\,, 
\end{eqnarray*}
where 
\[ 
 	B= \Bigg\{ f\in \mathcal{BV}:f( t)\in (1-\delta )\Big(
        t^{\omega +\delta  }\mu + M^{ -(\omega +\delta) }A \Big)
        \ \mbox{ for some } 0\le t\le 1\Bigg\} ,   
\]
and $ \mathcal{BV}$ is the space of measurable functions of bounded variation.
Applying the functional large deviation principle  in Theorem 2.2 in
\cite{ghosh:samorodnitsky:2009} we obtain 
$$
 	\limsup_{ n \rightarrow \infty   } \frac{ 1}{b_{nM} } P\big[
          N<T( a_{ n})\le nM \big] \le -  \inf_{ f\in \bar{ B}}I( f)    , 
$$
where the closure of $B$ is taken in the uniform topology, and 
\[ 
 	I( f)= \left\{  \begin{array}{ll}  \int\limits_{ 0}^{ 1}I_{
            l}\big( f^{ \prime }(t) \big) dt  &\mbox{ if }f\in
          \mathcal{AC}, f( 0)=0  \\ \infty & \mbox{
            otherwise.} \end{array}\right.     
\]
Clearly, 
$$
\bar{ B} = \Bigg\{ f\in \mathcal{BV}:f( t)\in (1-\delta )\Big(
        t^{\omega +\delta  }\mu + M^{ -(\omega +\delta) }\bar{A} \Big)
        \ \mbox{ for some } 0\le t\le 1\Bigg\}, 
$$
and so
\begin{equation}\label{e:FLDP.A}
	\limsup_{ n \rightarrow \infty   } \frac{ 1}{b_{nM} } P\big[
          N<T( a_{ n})\le nM \big] \le - \inf_{y\in
          \bar{A}}\inf_{0\leq t_0\leq 1} \inf_{f\in G_{y,t_0}}
\int\limits_{ 0}^{ 1}I_{
            l}\big( f^{ \prime }(t) \big) dt,
\end{equation}
where
$$
G_{y,t_0} = \Bigg\{ f\in \mathcal{AC}: 
f( t_0)= (1-\delta )\Big(
        t_0^{\omega +\delta  }\mu + M^{ -(\omega +\delta)
        }y\Bigr)\Bigg\}.
$$

Next, we notice that for every $f\in G_{y,t_0}$ we have by the
definition of the rate function $I_l$ in \eqref{e:LDP} and convexity, 
$$
\int\limits_{ 0}^{ 1}I_{
            l}\big( f^{ \prime }(t) \big) dt
= \int\limits_{ 0}^{ 1} \frac12 f^{ \prime }(t)^{ \prime } \Sigma^{-1} f^{ \prime
}(t)\, dt
$$
$$
\geq \int\limits_{ 0}^{t_0} \frac12 f^{ \prime }(t)^{ \prime } \Sigma^{-1} f^{ \prime
}(t)\, dt
$$
$$
\geq \frac{1}{2t_0} \left( \int\limits_{ 0}^{t_0} f^{ \prime}(t)\,
dt\right)^{ \prime } \Sigma^{-1} \left( \int\limits_{ 0}^{t_0} f^{ \prime}(t)\,
dt\right) 
$$
$$
= \frac{1}{2t_0} f(t_0)  \Sigma^{-1} f(t_0)
$$
$$ 
= \frac{1}{2t_0} (1-\delta)^2 \left[ \Bigl( t_0^{\omega +\delta
  }\mu + M^{ -(\omega +\delta)}y\Bigr)^{ \prime } \Sigma^{-1} 
\Bigl( t_0^{\omega +\delta
  }\mu + M^{ -(\omega +\delta)}y\Bigr)\right].
$$
Introducing the variable $c=\bigl( t_0M)^{-(\omega+\delta)}$, we obtain
$$
	\limsup_{ n \rightarrow \infty   } \frac{ 1}{b_{nM} } P\big[
          N<T( a_{ n})\le nM \big] 
$$
$$
\le 
-\inf_{c\geq M^{-(\omega+\delta)}}\inf_{y\in \bar{A}}
M^{1-2(\omega+\delta)} c^{1/(\omega+\delta)-2}\, (1-\delta)^2\, 
\frac12\, (\mu+cy)^{ \prime }\Sigma^{-1}(\mu+cy),
$$
and so for every $0<\delta<1$, 
$$
	\limsup_{ n \rightarrow \infty   } \frac{ 1}{b_{n} } P\big[
          N<T( a_{ n})\le nM \big]
$$
$$
\leq -M^{-2\delta}\, (1-\delta)^2 \inf_{c\geq M^{-(\omega+\delta)}}c^{1/(\omega+\delta)-2}
\inf_{y\in \bar{A}} \frac12\, (\mu+cy)^{ \prime }\Sigma^{-1}(\mu+cy).
$$
Letting $\delta\to 0$, and noticing that the closure of $A$ plays no
role in the right hand side above, we obtain \eqref{eq:al=1fin}  and,
hence, conclude the proof. 
\end{proof}

\begin{proof}[Proof of Theorem \ref{thm.sm.1a}.]
The proof of this theorem is very similar to that of Theorem
\ref{thm.sm.1}. Note that now $(b_n)$ is a regularly varying sequence
with 
exponent $\nu$. We establish the lower bound of this part of the
theorem in the same was as in Theorem \ref{thm.sm.1},
except that we are using a different rate in the large
deviation principle, as given in \eqref{e:rate.f}.

For the upper bound, we also proceed as in the proof of the upper
bound in Theorem \ref{thm.sm.1}, but now we use Lemma 3.5(iii) and the
appropriate part 
of Theorem 2.2 in \cite{ghosh:samorodnitsky:2009}. This gives us
\eqref{e:FLDP.A}, but this time the rate function $I_l$ scales
according to
$$
I_l(ax) = a^{\beta/(\beta-1)}I_l(x), \ a>0, x\in \bbr^d\,.
$$
Therefore, for every $f\in G_{y,t_0}$
$$
\int\limits_{ 0}^{ 1}I_{
            l}\big( f^{ \prime }(t) \big) dt
= \int\limits_{ 0}^{ 1} \bigl( \Lambda^h\bigr)^*\bigl(
f^\prime(t)\bigr) \, dt
$$
$$
\geq \frac{1}{t_0^{1/(\beta-1)}} \bigl( \Lambda^h\bigr)^*\bigl(
f(t_0)\bigr)
$$
$$
= \frac{1}{t_0^{1/(\beta-1)}} (1-\delta)^{\beta/(\beta-1)}
\bigl( \Lambda^h\bigr)^*\Bigl(t_0^{\omega +\delta
  }\mu + M^{ -(\omega +\delta)}y\Bigr)\,.
$$
Therefore,
$$
	\limsup_{ n \rightarrow \infty   } \frac{ 1}{b_{nM} } P\big[
          N<T( a_{ n})\le nM \big] 
$$
$$
\le 
-\inf_{c\geq M^{-(\omega+\delta)}}\inf_{y\in \bar{A}}
M^{(1-\beta(\omega+\delta))/(\beta-1)}
c^{(1/(\omega+\delta)-\beta)/(\beta-1)} (1-\delta)^{\beta/(\beta-1)}
\bigl( \Lambda^h\bigr)^*(\mu+cy),
$$
and so for every $0<\delta<1$, 
$$
	\limsup_{ n \rightarrow \infty   } \frac{ 1}{b_{n} } P\big[
          N<T( a_{ n})\le nM \big]
$$
$$
\leq -M^{-\beta\delta/(\beta-1)}\, (1-\delta)^{\beta/(\beta-1)}
\inf_{c\geq M^{-(\omega+\delta)}}c^{(1/(\omega+\delta)-\beta)/(\beta-1)} 
\inf_{y\in \bar{A}} \bigl( \Lambda^h\bigr)^*(\mu+cy). 
$$
Now we let $\delta\to 0$ and complete the proof. 
\end{proof}

\begin{proof}[Proof of Theorem \ref{thm:ma:rp:lm}.]  The lower bound
  is obtained as in \eqref{rp:lower:short}, with $b_n=n$ and
  $\omega=2-\alpha$, using the appropriate part of the large deviation
  principle in \eqref{e:LDP} and \eqref{e:rate.f}.

The proof of the upper bound for $\alpha=1$ proceeds, once again,
similarly to that 
of Theorem \ref{thm.sm.1}. Let $J(t) = \Lambda(t)$. By the
assumption of zero mean we know that, for some $c>0$, $J(t) \leq c\|
t\|^2$ for all $t$ in a neighborhood of the origin. Therefore, we can
still select $t\in\bbr^d$ as in Condition $\mathcal A$ 
and $ \epsilon >0$ such that $J(t)-t\mu +\epsilon <0$, and we conclude
that \eqref{eq:al=1ptlast} still holds. Furthermore, using part (ii)
of Theorem 2.4 in \cite{ghosh:samorodnitsky:2009}, we conclude that 
\eqref{e:FLDP.A} holds  as well. Note that for every $f\in G_{y,t_0}$
by the convexity of the function $ \Lambda^*$, 
\begin{equation} \label{e:bound.I}
\int\limits_{ 0}^{ 1}I_{
            l}\big( f^{ \prime }(t) \big) dt
= \int\limits_{ 0}^{ 1} \Lambda^*\bigl(
f^\prime(t)\bigr) \, dt
\end{equation}
$$
\geq t_0  \Lambda^*\bigl( t_0^{-1} f(t_0)\bigr)
=  t_0\Lambda^*\Bigl( t_0^{-1} (1-\delta )\big(
        t_0^{1 +\delta  }\mu + M^{ -(1
          +\delta)}y\bigl)\Bigr).
$$
The same argument as in the proof of the upper bound in Theorem
\ref{thm.sm.1} shows that for any fixed $0<\theta<1$, 
$$
 \inf_{y\in
          \bar{A}}\inf_{\theta\leq t_0\leq 1} \inf_{f\in G_{y,t_0}}
\int\limits_{ 0}^{ 1}I_{
            l}\big( f^{ \prime }(t) \big) dt \geq M^{-1} 
\inf_{ c>0}c^{-1}\inf_{\gamma\in \bar A}
\Lambda^{ *}\big(\mu+c\gamma\big)\,.
$$
On the other hand, under the assumptions of the theorem,
$\Lambda^\ast$ grows super-linearly fast as the norm of its argument
increases. Therefore, it follows from \eqref{e:bound.I} that 
$$
\lim_{\theta\to 0}  \inf_{y\in
          \bar{A}}\inf_{0<t_0<\theta} \inf_{f\in G_{y,t_0}}
\int\limits_{ 0}^{ 1}I_{
            l}\big( f^{ \prime }(t) \big) dt=\infty. 
$$
This proves the upper bound in the case $\alpha=1$. 

Next we consider the case $\alpha<1$. Fix $t\in G$, and choose $
0<\epsilon <t\mu - \Lambda _{ \alpha }(t)$. We start with recalling
that, by Lemma 3.6(i) in \cite{ghosh:samorodnitsky:2009}, 
\[ 
 	\frac{ 1}{k }\log E \big[  e^{tS_{ k}/\Psi_{ k} } \big]\le
        \Lambda _{ \alpha }(t)+\epsilon 
\]
for all $k$ large enough, say, $k\geq N$. In particular, 
$\sup_{ k\ge 1} E \big[  e^{t S_{ k}/\Psi_{ k}-kt\mu} \big]
<\infty$. Let $\delta>0$. Notice that
\begin{eqnarray}\label{eq:rp:lowerpart}
	\limsup_{ n \rightarrow \infty   } \frac{ 1}{n}
\log P \big[  T( a_{ n})\le n\delta  \big] 
& \le  & \limsup_{ n \rightarrow \infty   }  \frac{ 1}{n} 
\log \sum_{ k=1}^{ [ n\delta ]} P \big[  tS_{ k}-a_{ k}t\mu >a_{
    n}\inf_{\gamma\in A} t\gamma \big]
\nonumber \\
	& \le & \limsup_{ n \rightarrow \infty   }  \frac{ 1}{n} 
\log \sum_{ k=1}^{ [ n\delta ]} e^{ - n\Psi_{ n}/\Psi_{ k}\inf_{\gamma\in A} t\gamma} 
E \big[  e^{ tS_{ k}/\Psi_{ k}-kt\mu} \big]\nonumber\\
	& \le & -\inf_{\gamma\in A} t\gamma\limsup_{ n \rightarrow
  \infty   } \frac{ \Psi_{ n}}{\Psi_{ [ n\delta] } }= -\delta^{\alpha -1 }
\inf_{\gamma\in A} t\gamma \,.
\end{eqnarray}
Next, for $ n\ge N/\delta $, the same argument gives us 
\begin{eqnarray*}
	 P \big[ n\delta <  T( a_{ n}) < \infty  \big]
	 & \le &  \sum_{ k=[ n\delta ]+1}^{ \infty }
P \Big[  \frac{tS_{ k}}{\Psi_{ k}} >kt\mu+ \frac{ n\Psi_{ n}}{\Psi_{
      k} }\inf_{\gamma\in A} t\gamma  \Big]\\	 
	 & \le &  \sum_{ k=[ n\delta ]+1}^{ \infty }\exp \Big\{ 
-\frac{ n\Psi_{ n}}{\Psi_{ k} } \inf_{\gamma\in A} t\gamma 
+ k\big( \Lambda _{ \alpha }(t)-t\mu +\epsilon \big)   \Big\}\,.
\end{eqnarray*}
We break up the sum into pieces. By the monotonicity of the sequence $
\bigl(\Psi_{ n}\bigr)$ and the choice of $\epsilon$ we have for $i\geq
1$, 
\begin{eqnarray*}
	 &  & \sum_{ k=i[ n\delta] +1}^{ ( i+1)[ n\delta]  }\exp \Big\{ 
-\frac{ n\Psi_{ n}}{\Psi_{ k} } \inf_{\gamma\in A} t\gamma 
+ k\big( \Lambda _{ \alpha }(t)-t\mu +\epsilon \big)   \Big\} \\
	 & \le & n\delta \exp \Big\{
- \frac{ n\Psi_{ n}}{ \Psi_{ ( i+1)[ n\delta] } }\inf_{\gamma\in A} t\gamma 
+ ( i[ n\delta] +1)\big( \Lambda _{ \alpha }(t)-t\mu +\epsilon \big)
\Big\} \,. 
\end{eqnarray*}

Let $0<\eta<1-\alpha$. By the Potter bounds (see Proposition 0.8 in
\cite{resnick:1987}) there exists $ N_{ 1}\ge 1$  such that for every
$ n\ge N_{ 1}$ we have both 
\[ 
 	\frac{ \Psi_{ n}}{\Psi_{ ( i+1)[ n\delta ]} }\ge x_{ i,\delta
        }( \eta ):=( 1-\eta)\min  \left\{ \big( ( i+1)\delta  \big)^{
          \alpha-1 -\eta} , \big( ( i+1)\delta  \big)^{ \alpha
          -1+\eta} \right\}    
\]
and $(i[ n\delta ]+1)/n\geq i\delta ( 1-\eta)$. We conclude that for $
n>\max  \left\{ N/\delta ,N_{ 1} \right\}$ and $i\geq 1$, 
\begin{eqnarray*}
	 &  & \sum_{ k=i[ n\delta] +1}^{ ( i+1)[ n\delta]  }\exp
  \Big\{ -\frac{ n\Psi_{ n}}{\Psi_{ k} } \inf_{\gamma\in A} t\gamma +
  k\big( \Lambda _{ \alpha }(t)-\mu t+\epsilon \big)   \Big\} \\ 
	 & \le & n\delta \exp \Big\{- n \Big(x_{ i,\delta }( \eta )
  \inf_{\gamma\in A} t\gamma 
-  i\delta ( 1-\eta)\big( \Lambda _{ \alpha }(t)-\mu t+\epsilon \big)
  \Big)  \Big\}\,.  
\end{eqnarray*}
Denoting $ y_{ i}= x_{ i,\delta }( \eta )  \inf_{\gamma\in A} t\gamma
- i\delta ( 1-\eta)\big( \Lambda _{ \alpha }(t)-\mu t+\epsilon \big) $
and $ y^{ *}=\min_{ i\ge 1}y_{ i}$, we see that $ y^{ *}>0$ and that
$ y^{ *}=y_{ i^{ *}}$ for some $ i^{ *}\ge 1$. Therefore, for every $
n>\max  \left\{ N/\delta ,N_{ 1} \right\}$ we have
$$
P \big[ n\delta <  T( a_{ n}) < \infty  \big] \leq 
n\delta \exp \big\{ -n y^{ *} \big\}\sum_{ i=1}^{ \infty }\exp \big\{ -n
( y_{ i}-y^{ *}) \big\}
$$
and, therefore,
\begin{equation}\label{eq:bigpart}
	  \limsup_{ n \rightarrow \infty   }  \frac{ 1}{ n}\log P
          \big[ n\delta <  T( a_{ n}) < \infty  \big]   \le  -y^{
            *} \,.  
\end{equation}
Combining \eqref{eq:rp:lowerpart} and \eqref{eq:bigpart} we obtain 
\[ 
 	   \limsup_{ u \rightarrow \infty   } \frac{ 1}{a^{ \leftarrow
             }(u) }\log \rho( u)= \limsup_{ n \rightarrow \infty   }
           \frac{ 1}{n }\log P \big[  T( a_{ n})< \infty  \big] \le
           \max\big\{ -t^{ *}\delta ^{ \alpha -1},-y^{ *} \big\} \,. 
\]
Letting $ \epsilon $ and $ \eta$ decrease to $ 0$, we
conclude that
$$
 	   \limsup_{ u \rightarrow \infty   } \frac{ 1}{a^{ \leftarrow
             }(u) }\log \rho( u)
\leq -\min_{ i\ge 1}\Bigl( \bigl((i+1)\delta
\bigr)^{\alpha-1}\inf_{\gamma\in A} t\gamma -i\delta\bigl( \Lambda _{
  \alpha }(t)-\mu t\bigr)\Bigr)
$$
$$
\leq -\inf\limits_{ u> 0}  \left\{u^{ \alpha -1}\inf_{\gamma\in A} t\gamma
 -u\Big( \Lambda _{\alpha }(t)-t\mu \Big)  \right\}
+ \delta\Big( \Lambda _{\alpha }(t)-t\mu \Big)\,.
$$
Letting, finally, $\delta\to 0$ and optimizing over $t\in G$ completes
the proof. 
\end{proof}

\begin{proof}[Proof of Theorem \ref{thm:ma:rp:lm.1}.] 
The lower bound in the theorem is established in the same way as the
lower bound in Theorem \ref{thm.sm.1}, using the fact that in the
present theorem, the sequence $(b_n)$ is regularly varying with
exponent $\nu=2(\omega+\alpha)-3$, the large deviation principle
\eqref{e:rate.f}, and the fact that $(G_\Sigma)_\alpha =
C_{\alpha,2}G_\Sigma$. 

For the upper bound, we consider, once again, the cases $ \alpha <1$
and $ \alpha =1$ separately. In the case $\alpha<1$ we notice that the
sequence $(a_n/b_n)$ is regularly varying with the exponent 
$$
\omega-\nu = 3-2\alpha-\omega\geq 1-\alpha>0\,.
$$
Therefore, the argument used in the proof of the upper bound in the
case $\alpha<1$ in Theorem \ref{thm:ma:rp:lm} applies in this case as
well, resulting in 
$$
 	\limsup_{ u \rightarrow \infty } \frac{ 1}
        { b_{ a^{\leftarrow}(u)}} \log \rho(u) \le  -\sup_{t\in G}
        \inf_{u>0} \left\{u^{ -(\omega-\nu)}\inf_{\gamma\in A} t\gamma
        - u^\nu \Bigl( \frac12 C_{\alpha,2}\, t^\prime \Sigma t -
        t\mu\Bigr)\right\} . 
$$
The infimum over $u$ is achieved at 
$$
u = \left( \frac{\nu}{\omega-\nu}\frac{t\mu-\frac12 C_{\alpha,2} \,
  t^\prime \Sigma t}{\inf_{\gamma\in A} t\gamma}\right)^{
  -1/\omega }\,,
$$
and the upper bound in the case $\alpha<1$ is obtained by
substitution. 

The argument in the case $ \alpha =1$ is the same as the argument of
the corresponding case in Theorem \ref{thm:ma:rp:lm}. 
\end{proof}

\begin{proof}[Proof of Theorem \ref{thm:ma:rp:lm.1a}.] 
The lower bound in the theorem is, once again, established in the same
way as the 
lower bound in Theorem \ref{thm.sm.1}, using the fact that in the
present theorem, the sequence $(b_n)$ is regularly varying with
exponent $\nu= \bigl(\beta(w+\alpha-1)-1\bigr)/(\beta-1)$, the large
deviation principle \eqref{e:rate.f}, and the fact that $(\Lambda ^h)_\alpha =
C_{\alpha,\beta}\Lambda ^h$. 

We prove now the upper bound. Suppose first that $\alpha<1$ and
$\omega<\beta(1-\alpha)$. In this case $\omega-\nu>0$ and we use, once
again, the argument of the proof of the upper bound in the
case $\alpha<1$ in Theorem \ref{thm:ma:rp:lm}. This gives us this time 
$$
 	\limsup_{ u \rightarrow \infty } \frac{ 1}
        { b_{ a^{\leftarrow}(u)}} \log \rho(u) \le  -\sup_{t\in G^{(1)}}
        \inf_{u>0} \left\{u^{ -(\omega-\nu)}\inf_{\gamma\in A} t\gamma
        - u^\nu \Bigl(  C_{\alpha,\beta}\, \Lambda^h(t)- 
        t\mu\Bigr)\right\} . 
$$
The infimum over $u$ is achieved at 
$$
u = \left( \frac{\nu}{\omega-\nu}\frac{t\mu- C_{\alpha,\beta} \,
  \Lambda^h(t)}{\inf_{\gamma\in A} t\gamma}\right)^{
  -1/\omega }\,,
$$
and the required upper bound follows by substitution. 

Next, we suppose that $\alpha<1$ and $\omega>\beta(1-\alpha)$. The
proof is similar to that of the proof of the upper bound in the
case $\alpha<1$ in Theorem \ref{thm:ma:rp:lm}, but relies on Lemma
\ref{lem:suplim} below in addition to Lemma 3.6 in
\cite{ghosh:samorodnitsky:2009}. 

For $t\in\bbr^d$ and $u>0$ let 
$J_{ u}(t)=u^{ 1+(1-\alpha )\beta }C_{ \alpha,\beta  }\Lambda^h(t)$. 
Let $0<\delta<1$, and note that by Lemma \ref{lem:suplim}, for any
$t\in\bbr^d$ as in Condition \ref{cond:A}, 
\begin{eqnarray*}
	 &  & \limsup_{ n \rightarrow \infty  } \frac{ 1}{b_{ n} }
  \log P\big[ T(a_{ n})\le n\delta  \big]\\ 
	 & \le &   \limsup_{ n \rightarrow \infty  } \frac{ 1}{b_{ n}
  } \log \sum_{ k=1}^{ [ n\delta ]} P \big[  tS_{ k} >a_{ n}
    \inf_{\gamma\in A}t\gamma\big]\\  
	 & \le & \limsup_{ n \rightarrow \infty  } \frac{ 1}{b_{ n} }
  \log \sum_{ k=1}^{ [ n\delta ]} e^{ -b_{ n}\inf_{\gamma\in A}t\gamma} E\Big[ \exp \{   
    \frac{ b_{ n}}{a_{ n} }tS_{ k}  \} \Big] \\ 
	 & \le & -\inf_{\gamma\in A}t\gamma 
 + \limsup_{ n \rightarrow \infty  } \frac{ 1}{b_{ n} } \log \Bigl\{ n
 \delta \sup_{ k\le n\delta }
E\Big[ \exp \{    \frac{ b_{ n}}{a_{ n} }tS_{ k}  \} \Big]\Bigr\}
 = -\inf_{\gamma\in A}t\gamma + J_{ \delta }(t). 
\end{eqnarray*}
 Since $ J_{ \delta }(t)\to 0$ as $ \delta \to 0$ for every $t$, we
 see that 
$$
\lim_{\delta\to 0} \limsup_{ n \rightarrow \infty  } \frac{ 1}{b_{ n} }
  \log P\big[ T(a_{ n})\le n\delta  \big]
\leq -\inf_{\gamma\in A}t\gamma\,.
$$
Since we may replace $t$ by $ct$ for any $c>0$ without violating the
restrictions imposed by Condition \ref{cond:A}, we let $c\to\infty$ to
conclude that
\begin{equation} \label{e:last.small}
\lim_{\delta\to 0} \limsup_{ n \rightarrow \infty  } \frac{ 1}{b_{ n} }
  \log P\big[ T(a_{ n})\le n\delta  \big] = -\infty\,.
\end{equation}

Further, using Lemma 3.6 in \cite{ghosh:samorodnitsky:2009} the
argument used to prove \eqref{eq:al=1ptlast} applies, and gives us 
\begin{equation} \label{e:last.big}
\lim_{\delta\to 0} \limsup_{ n \rightarrow \infty  } \frac{ 1}{b_{ n} }
  \log P\big[ n\delta^{-1}\leq T(a_{ n})<\infty  \big] = -\infty\,.
\end{equation}

Next, fix $t\in G^{(2)}$. This means that we can choose $0< \epsilon < 1$ so
small that $J_{  u}(t)- u^\omega t\mu - \inf_{\gamma\in A}t\gamma +\epsilon<0$
for all $u>0$. For $0<\delta<1$ we have, as before,  
\begin{eqnarray*}
	&& \limsup_{ n \rightarrow \infty   } \frac{ 1}{b_{ n}}\log P
  \big[ n\delta  <  T( a_{ n}) <  n\delta^{-1}  \big]\\ 
	 & \le &  \limsup_{ n \rightarrow \infty   } \frac{ 1}{b_{ n}
  }\log \sum_{ k=[ n\delta ]+1}^{ [n\delta^{-1}] }P \big[  tS_{ k}-a_{ k}t\mu >a_{
    n}\inf_{\gamma\in A} t\gamma \big] \\ 
	 & \le & \limsup_{ n \rightarrow \infty   } \frac{ 1}{b_{ n} }
\log \sum_{ k=[ n\delta ]+1}^{ [n\delta^{-1}] }
\exp \Big\{ -b_{ n}\Big( \inf_{\gamma\in A}t\gamma 
+ \frac{ a_{ k}}{a_{ n} } t\mu \Big) \Big\}  \,
E \Big[  \exp\bigl\{ \frac{ b_{ n}}{a_{ n} }tS_{ k} \bigr\}\Big]  .
\end{eqnarray*}

Let $0<\eta<1$. By the Potter bounds there exists $ N_{ 1}\ge 1$
such that for  $ k,l\ge N_{ 1}$ 
\[ 
  	\frac{ a_{ k}}{ a_{ l}}\ge a_{k,l}(\eta) 
:=  (1-\eta )\min \bigg\{\Big( \frac{ k}{l } \Big)^{ \omega-\eta},
        \Big(\frac{ k}{l }\Big)^{ \omega +\eta }\bigg\}
\]
and for every $n\geq N_1$, $[n\delta]/n\geq (1-\eta)\delta$. For every
$ n>N_1/\delta $ and $i\geq 1$, 
\begin{eqnarray*}
	 &  & \sum_{ k=i[ n\delta] +1}^{ ( i+1) [n\delta]  } 
\exp \Big\{ -b_{ n}\Big( \inf_{\gamma\in A}t\gamma 
+ \frac{ a_{ k}}{a_{ n} } t\mu \Big) \Big\} \, 
E \Big[  \exp\{ \frac{ b_{ n}}{a_{ n} }tS_{ k} \bigr\}\Big] \\
	 & \le & \sum_{ k=i[ n\delta] +1}^{ ( i+1) [n\delta]  } 
\exp \Big\{ -b_{ n}\Big( \inf_{\gamma\in A}t\gamma 
+  a_{k,n}(\eta) t\mu \Big) \Big\}  \,
E \Big[ \exp\bigl\{  \frac{ b_{ n}}{a_{ n} }tS_{ k} \bigr\} \Big] \\
	 & \le & n\delta 
\exp \Big\{ -b_{ n}\Big( \inf_{\gamma\in A}t\gamma 
+  a_{i[n\delta ],n}(\eta) t\mu \Big) +\sup_{ k\le (i+1) [n\delta ]}
\log E \Big[ \exp\bigl\{  \frac{ b_{ n}}{a_{ n} }tS_{ k} \bigr\} \Big]
\Big\}\,.  
\end{eqnarray*}

By the choice of $n$, we known that for every $ i\ge 1$, $ a\big( [ in
  \delta ],n \big)(\eta)\ge (1-\eta )^{ \omega +\eta+1 }a(i\delta
,1)(\eta)$. Furthermore, by Lemma \ref{lem:suplim}, we can choose
$N_2$ so large that for all $n\geq N_2$, all $i=1,2,\ldots,
\delta^{-2}+1$, 
$$
\sup_{ k\le (i+1) [n\delta ]}
\log E \Big[ \exp\bigl\{  \frac{ b_{ n}}{a_{ n} }tS_{ k} \bigr\} \Big]
\leq b_n\bigl( J_{ (i+1) \delta }(t)+\epsilon \big).
$$
Therefore, for all $n\geq \max(N_1/\delta, N_2)$ and $i$ as above,
\begin{eqnarray*}
	 &  & \sum_{ k=i[ n\delta] +1}^{ ( i+1) [n\delta]  } 
\exp \Big\{ -b_{ n}\Big( \inf_{\gamma\in A}t\gamma 
+ \frac{ a_{ k}}{a_{ n} } t\mu \Big) \Big\} \, 
E \Big[  \exp\{ \frac{ b_{ n}}{a_{ n} }tS_{ k} \bigr\}\Big]  \\
	 & \le & n \delta \exp \Big\{ -b_{ n}\Big( \inf_{\gamma\in A}t\gamma 
+ (1-\eta )^{ \omega +\eta+1 }a(i\delta ,1)(\eta)t\mu\bigr) +b_n\bigl( J_{
  (i+1) \delta }(t)+\epsilon \big)  \Big\} . 
\end{eqnarray*}

We proceed as in the proof of the upper bound in the
case $\alpha<1$ in Theorem \ref{thm:ma:rp:lm}. Setting 
$$
y_{ i}= +\inf_{\gamma\in A}t\gamma + (1-\eta )^{ \omega +\eta+1
}a(i\delta ,1)(\eta)t\mu  -J_{  (i+1) \delta }(t)-\epsilon
$$
and $ y^{ *}=\min_{i}y_{ i\geq 1}$,
we proceed as in the above prove and conclude that 
\begin{equation}\label{eq:bigpart:huge}
	  \limsup_{ n \rightarrow \infty   }  \frac{ 1}{b_{ n} }\log
          P\big[  n\delta <T(a_{ n})< \infty  \big]  \le  -y^{ *}  . 
\end{equation}
Combining \eqref{eq:bigpart:huge}, \eqref{e:last.small} and
\eqref{e:last.big}, and letting first $\delta\to 0$, and then $\eta\to
0$ and $\epsilon\to 0$, we obtain 
\begin{eqnarray*}
	 \limsup_{ u \rightarrow \infty   } \frac{ 1}{b_{ a^{
               \leftarrow }(u)} }  \log \rho( u) 
	 & \le & \sup_{ u> 0}  \left\{- \inf_{\gamma\in A}t\gamma
         -u^{ \omega } t\mu + u^{1+\beta(1-\alpha)}
         C_{\alpha,\beta}\Lambda^h(t)  \right\} \,.
\end{eqnarray*}
The supremum is attained at 
\[ 
 	u=\Big( \frac{ 1+(1-\alpha )\beta }{\omega t\mu} C_{ \alpha
          ,\beta }\Lambda^h(t)\Big)^{ \frac{ 1}{\omega
            -(1+(1-\alpha )\beta ) }},      
\]
and the required upper bound is obtained by substitution and optimizing
over $t$. 

Finally, in the case $ \alpha =1$ the upper bound of the present
theorem can be obtained in the same way as in 
Theorem \ref{thm.sm.1}. 
\end{proof}

This section is concluded by a lemma needed for the proof of Theorem
\ref{thm:ma:rp:lm.1a}. 

\begin{lemma}\label{lem:suplim} 
Under the assumption $ R4$ with $\alpha<1$, for any $ \theta>0$ and $
t\in\bbr^d$ 
 \[ 
 	\lim_{ n \rightarrow \infty   } \frac{ 1}{b_{ n} }\sup_{ k\le
          \theta n }\log E \Big[\exp\Big\{ \frac{ b_{ n}}{a_{ n} }tS_{ k}
          \Big\} \Big]    \le u^{ 1+(1-\alpha )\beta }C_{ \alpha,\beta
        } \Lambda^h(t), 
\]
where $C_{\alpha,\beta}$ is given by \eqref{e:C}. 
\end{lemma}

\begin{proof} 
Observe that since
 the coefficients satisfy \eqref{longmemory}, there is $ N\ge 1$ such
 that $ \phi_{ i,n}>0$ for all $ i\in \mathbb{Z}$ and $ 
 n\ge N$. Using the fact that $ \Lambda (t)$ is increasing  along each
 ray emanating from the origin, we see that, if $ n\ge N/\theta$, 
\begin{eqnarray*}
	 \sup_{N\le k\le \theta n }\log E \Big[\exp\Big\{ \frac{ b_{
               n}}{a_{ n} }tS_{ k} \Big\} \Big] 
	 & = & \sup_{N\le k\le \theta n }\sum_{ i\in \mathbb{Z}} \Lambda
         \Big( t \frac{ b_{ n}}{a_{ n} } \phi_{ i,k} \Big) \\ 
	 & \le & \sup_{N\le k\le \theta n } \sum_{ i\in \mathbb{Z}}
         \Lambda \Big( t \frac{ b_{ n}}{a_{ n} } \big| \phi \big|_{
           i,k}  \Big) \\ 
	 & = &  \sum_{ i\in \mathbb{Z}} \Lambda \Big( t \frac{ b_{
             n}}{a_{ n} } \big| \phi \big|_{ i,[ \theta n]}  \Big), 
\end{eqnarray*}
where $ | \phi |_{ i,n}= | \phi_{ i+1} |+\cdots+ | \phi_{ i+n} |
$. Clearly, the sequence  $ (| \phi_{ i} |) $ is also balanced
regular varying and satisfies  
 \[ 
 	\frac{ | \phi_{ n} |  }{\psi(n) } \to p \mbox{ and } \frac{ |
          \phi_{ -n} |  }{ \psi(n) } \to q \ \ \ \ \mbox{ as } n\to
        \infty ,    
\] 
where $ \psi(\cdot)$ is as in \eqref{longmemory}.
With a minor modification of the proof of Lemma 3.6 in
\cite{ghosh:samorodnitsky:2009} we obtain, for any $ t\in\bbr^d$ and
$\theta>0$, 
\[ 
 	\lim_{ n \rightarrow \infty   }   \frac{ 1}{b_{ n} } \sum_{
          i\in \mathbb{Z}} \Lambda \Big( t \frac{ b_{ n}}{a_{ n} }
        \big| \phi \big|_{ i,[ un]}  \Big) =  u^{ 1+(1-\alpha )\beta }C_{ \alpha,\beta
        } \Lambda^h(t). 
\]
Since it is also easy to see that 
\begin{eqnarray*}
	\lim_{ n \rightarrow \infty   } \frac{ 1}{ b_{ n} } \sup_{
          k\le N} \log E \Big[  \exp\Big(  \frac{ b_{ n}}{a_{ n} }
          tS_{ k}\Big)  \Big]&=&0 \,,
\end{eqnarray*}
the proof is complete. 
\end{proof}

\end{section}

\begin{section}{Appendix} \label{sec:appendix}

In this section we state certain straightforward multivariate analogs
of the ruin probability estimates of \cite{nyrhinen:1994}. For
completeness we provide the argument.

Let $\bigl(Y_n,\, n\geq 1\bigr)$ be an $\bbr^d$-valued stochastic process. For
$n=1,2,\ldots$ and $t\in\bbr^d$ define $g_n(t) = n^{-1}\log Ee^{tY_n}$
and 
\begin{equation} \label{e:lim.trans}
g(t) = \limsup_{n\to\infty} g_n(t), \ t\in\bbr^d
\end{equation}
(these functions may  take the value $+\infty$). Let
$A\subset\bbr^d$ be a Borel set, and define 
\begin{equation} \label{e:sets}
{\mathcal C} = \bigl\{ t\in\bbr^d:\, \inf_{\gamma\in
  A}t\gamma>  0\bigr\}, \ \ 
{\mathcal D} = \bigl\{ t\in {\mathcal C}:\, \sup_{n\geq
  1}Ee^{tY_n}<\infty\bigr\}\,, 
\end{equation}
and
\begin{equation} \label{e:sets.E}
{\mathcal E} = \bigl\{ t\in\bbr^d:\, g \ \text{is finite in a
  neighborhood of $t$, exists as a limit at $t$,}
\end{equation}
$$
\text{and is   differentiable at $t$}\bigr\}, \ \ {\mathcal F} = \bigl\{ t\in
     {\mathcal E}:\,  \rho\, \nabla g(t)\in A^\circ\ \text{for some
       $\rho>0$}\bigr\}.
$$
\begin{theorem} \label{t:nyrh3.1}
(i) \ Suppose that there is $t_0\in {\mathcal C}$ such that $g(t_0)<0$. Then
$$
\limsup_{u\to\infty} \frac1u \log P\bigl( Y_n\in uA \ \ \text{for some
  $n=1,2,\ldots$}\bigr) \leq -\sup_{t\in {\mathcal D}}\inf_{\gamma\in A}t\gamma\,.
$$

(ii) \ For $t\in {\mathcal F}$, let $\eta(t) = \inf\{\eta>0:\, \eta\,
\nabla g(t)\in A^\circ\}$. Then
$$
\liminf_{u\to\infty} \frac1u \log P\bigl( Y_n\in uA \ \ \text{for some
  $n=1,2,\ldots$}\bigr) \geq \sup_{t\in {\mathcal F}}\eta(t)\bigl[
  g(t)-t\,\nabla g(t)\bigr]\,.
$$
\end{theorem}
\begin{proof}
(i) \ 
For $n=1,2,\ldots$ let $t\in \bbr^d$ be such that $g_n(t)<\infty$. Let
$Z_n$ be an $\bbr^d$-valued random vector such that
$$
P(Z_n\in B) = e^{-ng_n(t)}E\Bigl[ e^{tY_n}\one\bigl( Y_n\in
  nB\bigr)\Bigr],\ \text{$B\subseteq\bbr^d$ a Borel set.}
$$
Then
\begin{equation} \label{e:bound.n}
P(Y_n\in uA) = e^{ng_n(t)}E\Bigl[ e^{-ntZ_n}\one\Bigl( Z_n\in
  un^{-1}A\Bigr)\Bigr]
\leq \exp\bigl\{ ng_n(t) - u\inf_{\gamma\in A}t\gamma\bigr\}\,.
\end{equation} 
Fix $M=1,2,\ldots$. Using \eqref{e:bound.n} for $n\leq Mu$ and $t\in
{\mathcal D}$ gives us
$$
\sum_{n\leq Mu} P(Y_n\in uA) \leq (Mu)\, \sup_{n\geq  1}E^{tY_n}
\exp\bigl\{ - u\inf_{\gamma\in A}t\gamma\bigr\}\,.
$$
Taking a limit and optimizing over $t\in {\mathcal D}$ we obtain 
\begin{equation} \label{e:bound.s}
\limsup_{u\to\infty}u^{-1} \log\left( \sum_{n\leq Mu} P(Y_n\in
uA)\right) \leq -\sup_{t\in {\mathcal D}}\inf_{\gamma\in A}t\gamma\,.
\end{equation}
Next, using \eqref{e:bound.n} for $n>Mu$ and $t_0$ in the statement of
the theorem (which is possible for $u$ large enough) gives us for
large $u$ 
$$
\sum_{n> Mu} P(Y_n\in uA)\leq \sum_{n> Mu} e^{-\alpha n} \leq
Ce^{-\alpha Mu}\,,
$$
where $\alpha\in \bigl(g(t_0),0\bigr)$ and $C>0$ a constant. Therefore,
\begin{equation} \label{e:bound.l}
\limsup_{u\to\infty}u^{-1} \log\left( \sum_{n> Mu} P(Y_n\in
uA)\right) \leq -\alpha M\,.
\end{equation}
Combining \eqref{e:bound.s} with \eqref{e:bound.l} and letting
$M\to\infty$ we obtain the statement of part (i) of the theorem. 

For part (ii), let $t\in {\mathcal F}$, and let $\eta>0$ be such that
$\eta\, \nabla g(t)\in A^\circ$. Choose $\vep>0$ so that the
open ball $B\bigl(\eta\, \nabla g(t) ,\vep)$ lies completely within
$A$. Then for $u$ large 
enough, 
$$
P( Y_{[u\eta]}\in uA) \geq P\Bigl( Y_{[u\eta]}\in uB\bigl(\eta\,
\nabla g(t) ,\vep\bigr)\Bigr) 
\geq P\left( \frac{Y_{[u\eta]}}{[u\eta]} \in B\bigl(  \nabla g(t),
\vep/(2\eta)\bigr)\right)\,.
$$
On the other hand, for any $t\in {\mathcal E}$ and $\vep>0$, for all
$n$ large enough so that $g_n(t)<\infty$, we have
$$
P\Bigl( Y_n \in nB\bigl(  \nabla g(t), \vep\bigr)\Bigr)
=e^{ng_n(t)}E\Bigl[ e^{-ntZ_n}\one\Bigl( Z_n\in B\bigl(  \nabla g(t),
  \vep\bigr)\Bigr)\Bigr] 
$$
$$
\geq \exp\bigl\{ ng_n(t) - nt\nabla g(t) -
n\vep\| t\|\bigr\} P\Bigl( Z_n\in B\bigl(  \nabla g(t),
  \vep\bigr)\Bigr)\,,
$$
so that
$$
\liminf_{n\to\infty}n^{-1}\log P\Bigl( Y_n \in nB\bigl(  \nabla g(t),
\vep\bigr)\Bigr) 
$$
$$
\geq g(t) - t\nabla g(t) - \vep\| t\| + 
\liminf_{n\to\infty} n^{-1}\log P\Bigl( Z_n \in B\bigl(  \nabla g(t),
\vep\bigr)\Bigr)
$$
$$
=  g(t) - t\nabla g(t) - \vep\| t\|\,,
$$
since, as  is shown below, the last lower limit is equal to
zero. Therefore, for any $t\in {\mathcal F}$, $\eta>0$ as above  and
$\vep>0$ small enough, 
$$
\liminf_{u\to\infty} \frac1u \log P\bigl( Y_n\in uA \ \ \text{for some
  $n=1,2,\ldots$}\bigr)
$$
$$
\geq \liminf_{u\to\infty} \frac1u \log P( Y_{[u\eta]}\in uA) 
\geq \eta \bigl[ g(t) - t\nabla g(t) - \vep\| t\|\bigr]\,.
$$
Letting $\vep\to 0$, $\eta\to \eta(t)$, and optimizing over $t\in
{\mathcal F}$, we obtain the claim of part (ii) of the theorem. 

The proof of the theorem will be finished once we show that for every
$t\in {\mathcal E}$ and $\vep>0$, $P\bigl( Z_n \in B\bigl(  \nabla g(t),
\vep\bigr)\bigr)\to 1$ as $n\to\infty$. To this end, let $e_i$ be the
$i$th coordinate unit vector in $\bbr^d$, $i=1,\ldots, d$. Then
\begin{eqnarray*}
P\Bigl( Z_n \notin B\bigl(  \nabla g(t), \vep\bigr)\Bigr)
&\leq &\sum_{i=1}^d P\left( Z_ne_i\geq \frac{\partial g}{\partial
  y_i}(t) + \frac{\vep}{d}\right)\\
 &+& \sum_{ i=1}^{ d}P\left( Z_ne_i\leq \frac{\partial g}{\partial
  y_i}(t) - \frac{\vep}{d}\right)\,.
\end{eqnarray*}
Fix $i=1,\ldots, d$, and choose $r>0$ so small that
$g(t+re_i)<\infty$. Then $g_n(t+re_i)<\infty$ for all $n$ large
enough, and for such $n$ we have 
$$
P\left( Z_ne_i\geq \frac{\partial g}{\partial
  y_i}(t) + \frac{\vep}{d}\right)
= e^{-ng_n(t)}E\left[ \one\Bigl( Y_ne_i\geq n\frac{\partial g}{\partial
  y_i}(t) + n\frac{\vep}{d}\Bigr) e^{tY_n}\right]
$$
$$
\leq \exp\Bigl\{-ng_n(t) -rne_i\nabla g(t) - rn\vep/d\Bigr\}
E\left[ \one\Bigl( Y_ne_i\geq n\frac{\partial g}{\partial
  y_i}(t) + n\frac{\vep}{d}\Bigr) e^{(t+re_i)Y_n}\right]
$$
$$
\leq  \exp\Bigl\{ n\Bigl( g_n(t +re_i)-g_n(t)- re_i\nabla g(t)
-r\vep/d\Bigr)\Bigr\}\,.
$$
Therefore,
$$
\limsup_{n\to\infty} \frac1n \log P\left( Z_ne_i\geq \frac{\partial
  g}{\partial   y_i}(t) + \frac{\vep}{d}\right)
\leq g(t +re_i)-g(t)- re_i\nabla g(t) -r\vep/d\,.
$$
Since 
$$
g(t +re_i)-g(t)- re_i\nabla g(t) = o(r) \ \ \text{as $r\downarrow 0$,}
$$
this expression is negative for $r$ small enough, and so 
$$
P\left( Z_ne_i\geq \frac{\partial g}{\partial
  y_i}(t) + \frac{\vep}{d}\right) \to 0
$$
as $n\to\infty$ for every $i=1,\ldots, d$. A similar argument gives us 
\[ 
 	P\left( Z_ne_i\leq \frac{\partial g}{\partial
  y_i}(t) - \frac{\vep}{d}\right) \to 0   
\] as $n\to\infty$ for every $i=1,\ldots, d$ and the proof of the
theorem is complete. 
\end{proof}

\bigskip

\textbf{Acknowledgement.} We thank the anonymous referee for his/her helpful comments.

\end{section}

\bibliographystyle{elsart-harv}
\bibliography{c:/GenaFiles/texfiles/bibfile}

\begin{thebibliography}{15}
\expandafter\ifx\csname natexlab\endcsname\relax\def\natexlab#1{#1}\fi
\expandafter\ifx\csname url\endcsname\relax
  \def\url#1{\texttt{#1}}\fi
\expandafter\ifx\csname urlprefix\endcsname\relax\def\urlprefix{URL }\fi

\bibitem[{Arratia et~al.(1990)Arratia, Gordon, and
  Waterman}]{Arratia:1990p6736}
Arratia, R., Gordon, L., Waterman, M.~S., 1990. The {E}rdos-{R}{\'e}nyi law in
  distribution, for coin tossing and sequence matching. The Annals of
  Statistics 18~(2), 539--570.

\bibitem[{Asmussen(2003)}]{asmussen:2003}
Asmussen, S., 2003. Applied Probability and Queues, 2nd Edition. Springer, New
  York.

\bibitem[{Barbe and McCormick(2008)}]{barbe2008elf}
Barbe, P. and McCormick, W.P., 2008. An extension of a logarithmic form of {C}ramer's ruin theorem to some {FARIMA} and related processes. http://arxiv.org/pdf/0811.3460.

\bibitem[{Bingham et~al.(1987)Bingham, Goldie, and
  Teugels}]{bingham:goldie:teugels:1987}
Bingham, N., Goldie, C., Teugels, J., 1987. Regular Variation. Cambridge
  University Press, Cambridge.

\bibitem[{Brockwell and Davis(1991)}]{brockwell:davis:1991}
Brockwell, P., Davis, R., 1991. Time Series: Theory and Methods, 2nd Edition.
  Springer-Verlag, New York.

\bibitem[{Dembo and Zeitouni(1998)}]{dembo:zeitouni:1998}
Dembo, A., Zeitouni, O., 1998. Large Deviations Techniques and Applications,
  2nd Edition. Springer-Verlag, New York.

\bibitem[{Erd{\"o}s and R\'enyi(1970)}]{ErdRen:1970kx}
Erd{\"o}s, P., R\'enyi, A., 1970. On a new law of large numbers. Journal
  d'Analyse Math{\'e}matique 23~(1), 103--111.


\bibitem[{Gantert(1998)}]{gantert1998functional}
Gantert, N., 1998. Functional {E}rdos-{R}enyi laws for semiexponential random
  variables. The Annals of Probability 26~(3), 1356--1369.

\bibitem[{Gerber(1982)}]{gerber:1982}
Gerber, H., 1982. Ruin theory in linear models. Insurance: Mathematics and
  Economics 1, 177--184.

\bibitem[{Ghosh and Samorodnitsky(2009)}]{ghosh:samorodnitsky:2009}
Ghosh, S., Samorodnitsky, G., 2009. Large deviations and memory in moving
  average processes. Stochastic Processes and Their Applications 119, 534--561.
  
  \bibitem[{Gordon et~al.(1986)Gordon, Schilling, and
  Waterman}]{Gordon:1986p6959}
Gordon, L., Schilling, M.~F., Waterman, M.~S., 1986. An extreme value theory
  for long head runs. Probability Theory and Related Fields 72~(2), 279--287.

\bibitem[{H\"usler and Piterbarg(2004)}]{husler:piterbarg;2004}
H\"usler, J., Piterbarg, V., 2004. On the ruin probability for physical
  fractional brownian motion. Stochastic Processes and their Applications 113,
  315--332.

\bibitem[{H\"usler and Piterbarg(2008)}]{husler:piterbarg;2008}
H\"usler, J., Piterbarg, V., 2008. A limit theorem for the time of ruin in a
  gaussian ruin problem. Stochastic Processes and their Applications 118,
  2014--2021.

\bibitem[{Mansfield et~al.(2001)Mansfield, Rachev, and
  Samorodnitsky}]{mansfield:rachev:samorodnitsky:2001}
Mansfield, P., Rachev, S., Samorodnitsky, G., 2001. Long strange segments of a
  stochastic process and long range dependence. Annals of Applied Probability
  11, 878--921.

\bibitem[{Novak(1992)}]{Novak:1992p6961}
Novak, S.~Y., 1992. Longest runs in a sequence of {\$}m{\$}-dependent random
  variables. Probability Theory and Related Fields 91~(3-4), 269--281.

\bibitem[{Nyrhinen(1994)}]{nyrhinen:1994}
Nyrhinen, H., 1994. Rough limit results for level-crossing probabilities.
  Journal of Applied Probability 31, 373--382.

\bibitem[{Nyrhinen(1995)}]{nyrhinen:1995}
Nyrhinen, H., 1995. On the typical level crossing time and path. Stochastic
  Processes and their Applications 58, 121--137.

\bibitem[{Promislow(1991)}]{promislow:1991}
Promislow, S., 1991. The probability of ruin in a process with dependent
  increments. Insurance: Mathematics and Economics 10, 99--107.

\bibitem[{Rachev and Samorodnitsky(2001)}]{rachev:samorodnitsky:2001}
Rachev, S., Samorodnitsky, G., 2001. Long strange segments in a long range
  dependent moving average. Stochastic Processes and Their Applications 93,
  119--148.

\bibitem[{Resnick(1987)}]{resnick:1987}
Resnick, S., 1987. Extreme Values, Regular Variation and Point Processes.
  Springer-Verlag, New York.

\bibitem[{Samorodnitsky(2006)}]{samorodnitsky:2006LRD}
Samorodnitsky, G., 2006. Long range Dependence. Vol. 1:3 of Foundations and
  Trends in Stochastic Systems. Now Publishers, Boston.

\bibitem[{Vaggelatou(2003)}]{Vaggelatou:2003p6942}
Vaggelatou, E., 2003. On the length of the longest run in a multi-state
  {M}arkov chain. Statistics and Probability Letters 62~(3), 211--221.






\end{thebibliography}

\end{document}